\documentclass[11pt, 
a4paper,oneside,centertags,reqno, 
]{amsart}

\usepackage{endnotes}
\usepackage{latexsym}
\usepackage{amsmath}
\usepackage{amsfonts}
\usepackage{amssymb}
\usepackage{graphpap}
\usepackage{epsfig}
\usepackage{amscd}
\usepackage{amsthm}
\usepackage{enumerate}
\usepackage{eucal}

\usepackage[francais]{[babel}
\usepackage[deutsch]{[babel}
\usepackage[ansinew]{inputenc}
\usepackage{graphics}

\usepackage[pagebackref]{hyperref}
\usepackage[backrefs,non-compressed-cites,initials,nobysame]{amsrefs}

\usepackage{pdfsync}
\usepackage{color}
\usepackage{endnotes}
\usepackage{mathrsfs}

\usepackage{bbm}
\renewcommand{\Re}[0]{{\rm Re}\,}
\renewcommand{\Im}[0]{{\rm Im}\,}

\newcommand{\ca}[0]{\mathbbm{1}}

\newcommand{\C}[0]{{\mathbb C}}

\newcommand{\N}[0]{{\mathbb N}}

\newcommand{\pd}[0]{\partial}

\newcommand{\R}[0]{\mathbb{R}}

\newcommand{\cB}[0]{{\mathcal B}}
\newcommand{\cD}[0]{{\mathcal D}}
\newcommand{\cE}[0]{{\mathcal E}}
\newcommand{\cF}[0]{{\mathcal F}}

\newcommand{\cH}[0]{{\mathcal H}}
\newcommand{\cI}[0]{{\mathcal I}}

\newcommand{\cM}[0]{{\mathcal M}}

\newcommand{\cQ}[0]{{\mathcal Q}}
\newcommand{\cR}[0]{{\mathcal R}}

\newcommand{\cV}[0]{{\mathcal V}}

\newcommand{\oA}[0]{{\mathscr A}}
\newcommand{\oH}[0]{{\mathscr H}}
\newcommand{\oL}[0]{{\mathscr L}}
\newcommand{\oP}[0]{{\mathscr P}}

\newcommand{\Hmu}[0]{{\rm ({\bf H}$\gamma_{\infty}$)}}

\newcommand\bS{\mathbf{S}}

\newcommand{\nor}[1]{\| #1 \|}

\newcommand{\Nor}[1]{\left \| #1 \right \|}

\newcommand{\mn}[2]{\{ #1 : #2 \}}
\newcommand{\Mn}[2]{\left\{ #1 : #2 \right\}}
\newcommand{\sk}[2]{\langle #1 , #2\rangle}
\newcommand{\Sk}[2]{\left\langle #1 , #2\right\rangle}

\renewcommand{\div}[0]{{\rm div}\,}

\newcommand{\Dom}[0]{{\rm D}}
\newcommand{\Ran}[0]{{\rm R}}

\newcommand{\bena}[0]{{\mathbf 1}}

\newtheorem{theorem}{Theorem}
\newtheorem{lemma}[theorem]{Lemma}
\newtheorem{proposition}[theorem]{Proposition}
\newtheorem{corollary}[theorem]{Corollary}

\renewcommand\leq[0]{\leqslant}
\renewcommand\geq[0]{\geqslant}
\renewcommand\epsilon[0]{\varepsilon}
\renewcommand\theta[0]{\vartheta}

\newcommand\wrt{\,\text{\rm d}}
\renewcommand\mod[1]{\left\vert{#1}\right\vert}

\newcommand\norm[2]{{\left\Vert{#1}\right\Vert_{#2}}}

\newtheorem{prenotation}[theorem]{Notation}
\newenvironment{notation}%
{\begin{prenotation}\rm}{\end{prenotation}}

\newtheorem{preremark}[theorem]{Remark}  \newenvironment{remark}%
{\begin{preremark}\rm}{\end{preremark}}

\begin{document}

\title[Bounded holomorphic functional calculus]{Bounded holomorphic functional calculus for nonsymmetric Ornstein-Uhlenbeck operators}
\author[Carbonaro]{Andrea Carbonaro}
\author[Dragi\v{c}evi\'c]{Oliver Dragi\v{c}evi\'c}

\date{September 10, 2016}

\address{Andrea Carbonaro \\Universit\`a degli Studi di Genova\\ Dipartimento di Matematica\\ Via Dodecaneso\\ 35 16146 Genova\\ Italy }
\email{carbonaro@dima.unige.it}

\address{Oliver Dragi\v{c}evi\'c \\University of Ljubljana\\ Faculty of Mathematics and Physics\\ Jadranska 21, SI-1000 Ljubljana\\ Slovenia}
\email{oliver.dragicevic@fmf.uni-lj.si}

\begin{abstract}
We study bounded holomorphic functional calculus for nonsymmetric infinite dimensional Ornstein-Uhlenbeck operators $\oL$. We prove that if $-\oL$ generates an analytic semigroup on $L^{2}(\gamma_{\infty})$, then $\oL$ has bounded holomorphic functional calculus on $L^{r}(\gamma_{\infty})$, $1<r<\infty$, in any sector of angle $\theta>\theta^{*}_{r}$, where $\gamma_{\infty}$ is the associated invariant measure and $\theta^{*}_{r}$ the sectoriality angle of $\oL$ on $L^{r}(\gamma_{\infty})$. The angle $\theta^{*}_{r}$ is optimal. In particular our result applies to any nondegenerate finite dimensional Ornstein-Uhlenbeck operator, with dimension-free estimates. 
\end{abstract}

\maketitle

\section{Introduction}
For every $\theta\in (0,\pi)$ we define the open sector
$$
\bS_{\theta}=\{z\in\C\setminus\{0\}:\ |\arg z|<\theta\}
$$
and we denote by $H^\infty(\bS_{\theta})$ the algebra of all bounded holomorphic functions on $\bS_\theta$. If $m\in H^{\infty}(\bS_{\theta})$ we set $\|m\|_{\theta}=\sup\{|m(z)|: z\in\bS_{\theta}\}$.

If $\oA$ is a linear operator on a complex Banach space $X$ we denote, respectively, by $\sigma(\oA)$, $\Dom(\oA)$, ${\rm R}(\oA)$ and ${\rm N(\oA)}$ the spectrum, the domain, the range and the null-space of $\oA$. If $\oA$ is bounded then $\|\oA\|$ stands for its operator norm. We denote by $\cB(X)$ we class of all bounded operators on $X$.

Let $0\leq\theta<\pi$. We say that a densely defined closed operator $\oA$ on a complex Banach space $X$ is {\it sectorial of angle} $\theta$ if $\sigma(\oA)\subseteq \overline{\bS}_{\theta}$ and for all $\theta'\in (\theta,\pi)$ we have 
$$
\sup_{z\in\C\setminus\bS_{\theta'}}|z|\|(\oA-zI)^{-1}\|<\infty\,.
$$ 
In such a case the number 
$$
\omega(\oA):=\inf\mn{\theta\in[0,\pi)}{\oA \text{ is sectorial of angle }\theta }
$$
is called the {\it sectoriality angle} of $\oA$. Operators which are sectorial of some angle in $[0,\pi)$ will simply be called {\it sectorial}.

Suppose that $\oA$ is a one-to-one sectorial operator with dense range on a complex Banach space $X$ (by \cite[Theorem 3.8]{CDMY}, on reflexive Banach spaces every one-to-one sectorial operator has dense range). Then, if $\omega(\oA)<\theta<\pi$ and $f\in H^{\infty}(\bS_{\theta})$, we may define the closed, possibly unbounded operator $f(\oA)$ in such a way that, for $\alpha\in\C$ and $g\in H^{\infty}(\bS_{\theta})$,
$$
\aligned
\alpha f(\oA)+g(\oA)& =(\alpha f+g)(\oA)|_{\Dom(f(\oA))\cap\Dom(g(\oA))}\\ f(\oA)g(\oA)&=(fg)(\oA)|_{\Dom(g(\oA))\cap\Dom(fg(\oA))}\,.
\endaligned
$$ 
We refer the interested reader to \cite{Mc, CDMY, Haase} for an exhaustive treatment of this subject.

Suppose furthermore that $\theta\in(\omega(\oA),\pi)$. We say that $\oA$ has a {\it bounded $H^{\infty}(\bS_{\theta})$-calculus} if $m(\oA)\in\cB(X)$ whenever $m\in H^{\infty}(\bS_{\theta})$ and there exists $C>0$ such that
$$
\|m(\oA)\|\leq C\|m\|_{\theta},\quad \forall\  m\in H^{\infty}(\bS_{\theta}). 
$$ 
We say that $\oA$ has a {\it bounded $H^{\infty}$-calculus} if it has a bounded $H^{\infty}(\bS_{\theta})$-calculus for some $\theta>\omega(\oA)$. We define
$$
\omega_{H^{\infty}}(\oA)=\inf\{\theta\in (0,\pi):\ \oA\ \text{ has a bounded}\ H^{\infty}(\bS_{\theta})\text{-calculus}\},
$$
with the convention that $\omega_{H^{\infty}}(\oA)=+\infty$ if $\oA$ does not have a bounded $H^{\infty}$-calculus. It follows from definitions that for a sectorial operator $\oA$ we always have 
\begin{equation}\label{eq: *}
0\leq \omega(\oA)\leq\omega_{H^{\infty}}(\oA).
\end{equation}

It is an interesting and widely studied problem whether a sectorial operator has a bounded $H^{\infty}$-calculus \cite{Mc,McY,CDMY,KW}.
A. McIntosh \cite{Mc} proved that if $X=\cH$ is a complex Hilbert space and $\oA$ has a bounded $H^{\infty}$-calculus on $\cH$, then $\omega_{H^{\infty}}(\oA)=\omega(\oA)$. 
It was shown  by N. Kalton in \cite{Kalton} that this is not longer true in arbitrary Banach spaces (see also \cite[p. 27]{CDMY}), but it is still an open problem whether $\omega_{H^{\infty}}(\oA)=\omega(\oA)$ for sectorial operators with a bounded $H^{\infty}$-calculus on Lebesgue spaces.

Therefore in this context it becomes of interest to explicitly determine the angle $\omega_{H^{\infty}}(\oA)$, even for special classes of operators such as generators of semigroups with kernel bounds (see, for example, \cite{DM,DR,BK}), or such as generators of contraction semigroups on Lebesgue spaces which now we describe in more detail.

Let $(\Omega,\mu)$ be a $\sigma$-finite measure space. We say that $(T(t))_{t>0}$ is a {\it contraction semigroup} on $(\Omega,\mu)$ if $T(t)$ is a contraction on $L^{r}(\mu)$ for every $t>0$ and $r\in[1,\infty]$, and $(T(t))_{t>0}$ is strongly continuous on $L^{r}(\mu)$ for $1\leq r<\infty$ and weak* continuous on $L^{\infty}(\mu)$. Denote by $-\oA_{r}$ its generator on $L^{r}(\mu)$, $1< r<\infty$. 
A contraction semigroup is called {\it symmetric} if $T(t)$, $t>0$, are self-adjoint on $L^{2}(\mu)$. 
It is known that every contraction semigroup is subpositive \cite{CK} and \cite[Theorems 4.1.2, 4.1.3]{KRE}, therefore it has a dilatation to a group \cite[pp. 737-738]{Fe}. It follows from the Coifman-Weiss transference principle \cite{CW, CRW} that $\omega_{H^{\infty}}(\oA_{r})\leq \pi/2$, for every $r\in(1,\infty)$; see \cite[Theorem 2]{cowling} for the symmetric case and \cite[Theorem 2]{Duong} for the general case.

In this picture, symmetric contraction semigroups deserve a special attention. For generators of symmetric Markovian semigroups the inequality $\omega_{H^{\infty}}(\oA_{r})\leq \pi/2$ was originally proved by E. M. Stein \cite[Corollary 3, p. 121]{stein}, 
while M. Cowling \cite[Theorem 2]{cowling} combined the Coifman-Weiss transference mentioned above with a complex interpolation argument and proved that $\omega_{H^{\infty}}(\oA_{r})\leq \pi|1/2-1/r|$, $1<r<\infty$, for all generators of symmetric contraction semigroups. Cowling's result has been improved by P. C.~Kunstmann and \v{Z}.~\4trkalj \cite{KS} in the special case of sub-Markovian semigroups and by C. Kriegler \cite[Remark 2]{KR} in the case of symmetric contraction semigroups.
\medskip

For $\rho\in (-\pi/2,\pi/2)$ we will throughout the paper use the convention
$$
\rho^*=\pi/2-\rho\,.
$$
Moreover, for $r\in(1,\infty)$ we set $\phi_r=\arccos|1-2/r|$. Note that
$$
\phi_r^*= \arctan\frac{|r-2|}{2\sqrt{r-1}}.
$$
The two authors of the present paper proved in \cite{CD-mult} the optimal bound 
$$
\omega_{H^{\infty}}(\oA_{r})\leq \phi_r^*,\quad 1<r<\infty,
$$
for all generators of {\sl symmetric} contraction semigroups (see \cite[Theorem 1]{CD-mult} for a more accurate statement of this result).

This ``universal'' multiplier theorem cannot be improved, because as a consequence of a result by J. B. Epperson \cite{E} and of inequality \eqref{eq: *} one yields, for all $1<r<\infty$,
$$
\omega_{H^{\infty}}(\oL^{ou}_{r})\geq\omega(\oL^{ou}_{r})=\phi_r^*,
$$
where $\oL^{ou}$ denotes the symmetric finite dimensional Ornstein-Uhlenbeck operator with diffusion matrix $Q=I$ and drift matrix $A=I$, that will be defined in \eqref{d: ou}.
See also \cite[Theorem 2]{fcou}, \cite[Theorem 2.2]{HMM} and \cite{MMS} for a sharp multiplier result for $\oL^{ou}$.

Note, however, that for some generators of symmetric contraction semigroups it may happen that a sharper bounded functional calculus is available. See, for example, \cite{Mihlin1, Mihlin2, H, CS}.

In this paper we study optimal bounded $H^{\infty}$-calculus for a specific subclass of generators of nonsymmetric contraction semigroups: the class formed by Ornstein-Uhlenbeck operators on separable real Banach spaces which generate analytic semigroups with respect to associated invariant measures. They will be defined in Section~\ref{s: 2} in the nondegenerate finite dimensional case, and in Section~\ref{I: s: OU} in the general case.

\subsection*{Main result} 
Let $-\oL$ denote the generator of an analytic Ornstein-Uhlenbeck semigroup. We prove in Theorem~\ref{t: fc principal} and Theorem~\ref{I: t: fc principal} that, for $r\in(1,\infty)$,
\begin{equation}\label{eq: **}
\omega_{H^{\infty}}(\oL_{r})=\omega(\oL_{r})\,.
\end{equation}
The sectoriality angle of $\oL_{r}$ was calculated by R. Chill, E. Fa\v{s}angov\'a, G. Metafune and D. Pallara  \cite[Theorem 2 and Remark 6]{CFMP} in the nondegenerate finite-dimensional case, and by J. Maas and J.M.A.M. van Neerven in the general case \cite[Theorem 3.4 and Theorem 3.5]{MvN}. They proved that,  for $r\in(1,\infty)$,
$$
\omega(\oL_{r})=\arctan\frac{\sqrt{(r-2)^{2}+r^{2}(\tan\theta^{*}_{2})^{2}}}{2\sqrt{r-1}},
$$
where $\theta^{*}_{2}=\omega(\oL_{2})$; see Proposition~\ref{p: properties of L2}, Remark~\ref{r: optimal} and  Proposition~\ref{I: p: OU on Lp}. Since for symmetric Ornstein-Uhlenbeck operators one has $\theta^{*}_{2}=0$, we recover in this particular case the above-mentioned Epperson's result \cite{E} and the sharp angle of \cite[Theorem 1]{CD-mult}. 

As far as we are aware, our result is the first example of an explicit calculation of the functional calculus angle for any nonsymmetric Ornstein-Uhlenbeck operator.

\subsection*{Outline of the proofs} 
We now briefly describe the technique we utilise for proving equality \eqref{eq: **}. Our approach is based on \cite{CD-mult}, which was the first case of Bellman functions being applied for the study of spectral multipliers.

By a general result of M. Cowling, I. Doust, A. McIntosh and  A. Yagi \cite[Theorem 4.6 and Example 4.8]{CDMY}, in order to prove \eqref{eq: **} it suffices to establish certain bilinear estimates, of the type \eqref{eq: fc  bilinear}, involving $\oL_{r}$ and the  associated Ornstein-Uhlenbeck semigroup with complex time $te^{\pm i\theta}$ for $0\leq\theta<
\omega(\oL_{r})^*$. The classical approach for proving \eqref{eq: fc  bilinear} is based on square functions; we choose a different approach, avoiding square functions and dealing with bilinear integrals directly. 

Namely, in Theorem~\ref{t: heat flow}, by using the so-called Nazarov-Treil Bellman function $\cQ$ (defined in Section~\ref{s: Nazarov-Treil}) and a heat-flow argument, we reduce the proof of the bilinear estimates \eqref{eq: fc  bilinear}
to the verification of an integral condition \eqref{eq: heat flow infinitesimal} involving the Bellman function $\cQ$, the generator $\oL_{r}$ and its adjoint.
It turns out that the integral condition \eqref{eq: heat flow infinitesimal} of Theorem~\ref{t: heat flow} is nothing but an extension of the classical Lumer-Phillips condition for contraction semigroups in $L^{r}$ spaces.

We prove that the Ornstein-Uhlenbeck operator $\oL_{r}$ and the Bellman function $\cQ$ satisfy the above-mentioned integral condition \eqref{eq: heat flow infinitesimal}: in Section~\ref{s: proof of mult thm} we do this in the nonndegenerate finite dimensional case, and in Section~\ref{I: s: proof of mult thm} in the general case.

The principal ingredient of these proofs is a new convexity property of $\cQ$ stated in Theorem~\ref{t: convexity} and proved in Section~\ref{s: proof of convexity thm}.

Among other results, for the proof of Theorem~\ref{t: convexity} we utilise in Proposition~\ref{p: convexity of Fr} the calculation of the analyticity angle of the Ornstein-Uhlembeck semigroup done by Chill, Fa\v{s}angov\'a,  Metafune and Pallara in \cite[Theorem 2]{CFMP} and \cite[Theorem 1.1]{CFMP2}.

The reason why the above-described procedure leads to the identification $\omega_{H^{\infty}}(\oL_{r})=\omega(\oL_{r})$  
is that the Ornstein-Uhlenbeck operators enjoy the property that the angle $\omega(\oL_{r})^*$ coincides with the contractivity  angle on $L^{r}$ of the associated semigroup (see Proposition~\ref{p: OU on Lp}, Remark~\ref{r: optimal} and Proposition~\ref{I: p: OU on Lp}).

Both in finite dimension (Theorem~\ref{t: fc principal}) and in infinite dimension (Theorem~\ref{I: t: fc principal}) we utilise a representation of the Ornstein-Uhlenbeck operator $\oL$ in divergence form. 
While in finite dimension this representation follows from the definition (see \eqref{d: ou} and \eqref{eq: L Divform}), the analogous formula \eqref{I: eq: Divform} in infinite dimension is more delicate, and it was proved in \cite{Fu} and \cite{BRS} in the Hilbert space setting and extended to the Banach space setting in \cite{MvN}.
Once we get the divergence-form representation \eqref{I: eq: Divform} of $\oL$, the proof of Theorem~\ref{I: t: fc principal} becomes a natural generalisation of that of Theorem~\ref{t: fc principal}.

The theory of infinite dimensional Ornstein-Uhlenbeck operators is very rich; see, for example, \cite{DPZ,CG,Goldys,GGvN,vN,GvN,M} and the references therein.
Therefore, with the aim of highlighting the role played by Bellman function in our proofs, we decided to keep the finite dimensional case separated from the infinite dimensional one.

\section{Notation}
For each $k\in [0,\infty]$ we denote by $C^{k}_{b}(\R^{n})$ the class of bounded complex functions on $\R^{n}$ with bounded and continuous partial derivatives up to the order $k$. If $f,g$ are complex functions on some sets $X,Y$, respectively, then $f\otimes g$ is the function on $X\times Y$ mapping $(x,y)\mapsto f(x)g(y)$. 

We set $\C_+=\{z\in\C:\ \Re z>0\}$. If $\Omega\subset\C$ we denote its closure by ${\overline\Omega}$.

If $\oH$ is a complex Hilbert space, $\sk{\cdot}{\cdot}_{\oH}$ denotes the inner product on $\oH$ and $\oA$ is a linear operator on $\oH$, then we denote by $W(\oA)$ the numerical range of $\oA$; i.e. we set
$$
W(\oA)=\{\sk{\oA h}{h}_{\oH}:\ h\in \Dom(\oA),\ \norm{h}{\oH}=1\}.
$$ 
If $T\in\cB(\oH)$, then we denote respectively by $T_{{\bf s}}$ and $T_{{\bf a}}$ the symmetric and the antisymmetric part of $T$; i.e.
$$
T_{{\bf s}}=\frac{T+T^{*}}{2},\quad T_{{\bf a}}=\frac{T-T^{*}}{2},
$$
where $T^{*}$ is the adjoint of $T$ with respect to $\sk{\cdot}{\cdot}_{\oH}$. We have a pair of simple yet useful identities:
\begin{equation}
\label{eq: sim and antisim}
\Re\sk{T\xi}{\xi}=\sk{T_{{\bf s}}\xi}{\xi}
\hskip 30pt \text{and}\hskip 30pt 
i\Im\sk{T\xi}{\xi}=\sk{T_{{\bf a}}\xi}{\xi}
\,.
\end{equation}

If $\cH$ is a real Hilbert space and $T\in\cB(\cH)$, we denote, respectively, by $\cH^\C$ and $T^{\C}$ the complexification of $\cH$ and $T$. If there is no risk of ambiguity, we only write $T$ for denoting the complexification of $T$.

For $n\in\N=\{1,2,\dots\}$, we denote by $\C^{n,n}$ the space of all complex $n\times n$ matrices, and by $\R^{n,n}$ its subspace consisting of real $n\times n$ matrices. We canonically identify matrices in $\C^{n,n}$ with operators acting on the complex Hilbert space $\C^{n}$ endowed with the scalar product
$$
\sk{z}{w}_{\C^n}=\sum_{j=1}^nz_j\overline w_j.
$$
If $(\Omega,\mu)$ is a $\sigma$-finite measure space, then the associated complex and real Lebesgue spaces are denote respectively by $L^{r}(\mu)$ and $L^{r}_{\R}(\mu)$, $1\leq r\leq\infty$.

\section{Finite dimensional nonndegenerate Ornstein-Uhlenbeck operators}
\label{s: 2} 
Let $Q,A\in \R^{n,n}$ be such that $Q$ is symmetric and positive definite and $\sigma(A)\subset \C_{+}$. For each $t>0$ set $S(t)=e^{-tA}$ and define the matrices
$$
Q_{t}=\int^{t}_{0}S(u)QS^{*}(u)\wrt u,\quad Q_{\infty}=\int^{\infty}_{0}S(u)QS^{*}(u)\wrt u.
$$
Assumptions on $Q$ and $A$ ensure that $Q_{t}$ and $Q_{\infty}$ are well defined, symmetric and positive definite. A simple calculation (see for example \cite[Lemma 2.1]{MPRS}) shows that $Q_{\infty}$ solves the Lyapunov equation

\begin{equation}\label{eq: fund identity}
AQ_{\infty}+Q_{\infty}A^{*}=Q.
\end{equation}
For each $t\in (0,\infty]$ we denote by $\gamma_{t}$ the centered Gaussian measure on $\R^{n}$ with covariance matrix $Q_{t}$; i.e. we set
$$
\wrt \gamma_{t}(x)=\frac{1}{(2\pi)^{n/2}({\rm det}Q_{t})^{1/2}}\exp\left(-\frac{\sk{Q^{-1}_{t}x}{x}}{2}\right)\wrt x,\quad t\in(0,\infty].
$$
The Ornstein-Uhlenbeck semigroup $(T(t))_{t>0}$ associated with $(S(t))_{t>0}$ and $Q$ is defined on, say, $C_{b}(\R^{n})$ by the Kolmogorov's formula 
$$
T(t)f(x)=\int_{\R^{n}}f(S(t)x+y)\wrt\gamma_{t}(y),\quad x\in\R^{n}\,.
$$
It is well-known that the measure $\gamma_{\infty}$ is invariant under the action of $T(t)$, $t>0$, and that $(T(t))_{t>0}$ extends to a positivity preserving semigroup of contractions on $L^{r}(\gamma_{\infty})$, $1\leq r\leq\infty$, which is strongly continuous for $1\leq r<\infty$ and weak* continuous for $r=\infty$ (see, for example, \cite[Chapter 9]{lorenzi} and the references therein).

We will denote by $-\oL_{r}$ the generator of $(T(t))_{t>0}$ on $L^{r}(\gamma_{\infty})$, $1\leq r<\infty$. It is known that $C^{\infty}_{c}(\R^{n})$ is a core for $\oL_{r}$; see \cite{MPP,MPRS} and \cite[Lemma~9.3.13]{lorenzi}. The action of the Ornstein-Uhlenbeck operator on $C^{\infty}_{c}(\R^{n})$ is explicitly given by the formula 

\begin{equation}\label{d: ou}
\oL f(x)=-\frac{1}{2}\div(Q\nabla f)(x)+\sk{\nabla f(x)}{Ax}_{\C^{n}},
\quad x\in\R^{n},\quad f\in C^{\infty}_{c}(\R^{n}).
\end{equation}
We call $Q$ and $A$, respectively, the {\it diffusion} and the {\it drift} matrix of $\oL$. The usual symmetric finite-dimensional Ornstein-Uhlenbeck operator is obtained from this scheme by choosing $Q=A=I$.

\medskip
Denote by $\nabla^{*}_{\infty}$ the formal adjoint of $\nabla$ with respect to the scalar product in $L^{2}(\gamma_{\infty})$. Then for every $\omega\in C^{\infty}_{c}(\R^{n},\C^{n})$ we have
$$
\nabla^{*}_{\infty}\omega(x)=-{\rm div}\omega(x)+\Sk{\omega(x)}{Q^{-1}_{\infty}x}_{\C^{n}}
,\quad x\in\R^{n}.
$$
From this and the identity \eqref{eq: fund identity}, one rapidly sees that
\begin{equation}\label{eq: L Divform}
\oL f=\nabla^{*}_{\infty}(Q_{\infty}A^{*}\nabla f),\quad f\in C^{\infty}_{c}(\R^{n}).
\end{equation}
Set
$$
B=Q_{\infty}A^{*}.
$$
Identity \eqref{eq: fund identity} reads as $B_{{\bf s}}=Q/2$, so that
\begin{equation}\label{eq: accretive B}
\Re\sk{B\xi}{\xi}=\sk{B_{{\bf s}}\xi}{\xi}=\frac{1}{2}\sk{Q\xi}{\xi}\geq \lambda |\xi|^{2},\quad \xi\in\C^{n},
\end{equation}
for some $\lambda>0$. Therefore, $B$ is a strictly accretive real matrix.

\subsection*{Analitcity of the Ornstein-Uhlenbeck semigroup} 
Consider the sesquilinear form on $L^{2}(\gamma_{\infty})$ defined by
\begin{equation}\label{eq: sesquilinear form}
\mathfrak{a}(f,g)=\int_{\R^{n}}\sk{B\nabla f}{\nabla g}\wrt\gamma_{\infty},\quad \Dom({\mathfrak a})=W^{1,2}(\gamma_{\infty}),
\end{equation}
where
$$
W^{1,2}(\gamma_{\infty})=
\Mn{f\in L^{2}(\gamma_{\infty})}{\partial_{i}f\in L^{2}(\gamma_{\infty}),\ i=1,\dots,n}.
$$
It follows from \eqref{eq: accretive B} that  the form $\mathfrak a$ is densely defined, closed, continuous and accretive. Therefore, by means of the theory of sesquilinear forms \cite{K,O}, $\mathfrak a$ defines an accretive operator on $L^{2}(\gamma_{\infty})$, that we temporarily denote by $(L_{2}, \Dom(L_{2}))$. It is known that, in fact, $L_{2}=\oL_{2}$; see \cite[Theorem 1.2 in Appendix A]{Eberle}.

\begin{remark}
\label{r: symmetric OU}
It is known \cite[Proposition 9.3.10]{lorenzi} that 
$\oL_{2}$ is self-adjoint if and only if the matrix $B=Q_{\infty}A^{*}$ is symmetric. By identity \eqref{eq: fund identity}, this is equivalent to $B=Q/2$.
\end{remark}

Inequality \eqref{eq: accretive B} implies that the numerical range of $B$ is contained in some sector of angle less than $\pi/2$. Let $\theta^{*}_{2}=\theta^{*}_{2}(B)\in [0,\pi/2)$ be the {\it smallest} such angle, 
\label{t2} i.e. $\overline\bS_{\theta^*_2}$ is the smallest closed sector containing $W(B)$. It follows from the definition of Ornstein-Uhlenbeck operator by means of the sesquilinear form $\mathfrak a$ that $W(\oL_{2})\subseteq W(B)\subseteq\overline{\bS}_{\theta^{*}_{2}}$. Consequently, by Lumer-Phillips theorem
the semigroup $(T(t))_{t>0}$ extends to an analytic contraction semigroup on $L^{2}(\gamma_{\infty})$ in the sector $\bS_{\theta_{2}}$; see, for example, \cite[Theorem~1.54]{O}.
It was proved in \cite[Remark 2 and Remark 6]{CFMP} that 
\begin{equation}\label{eq: theta*2}
\theta^{*}_{2}=\arctan\norm{Q^{-1/2}(B-B^{*})Q^{-1/2}}{},
\end{equation}
and that $\theta^{*}_{2}$ coincides with the {\it spectral angle} of $Q^{-1/2}BQ^{-1/2}$ \cite[p. 705 and 708]{CFMP},
that is, $\theta^{*}_{2}$ is the smallest angle $\varphi$ such that $\sigma(Q^{-1/2}BQ^{-1/2})$ is contained in $\overline{\bS}_{\varphi}$. 
See also \eqref{eq: theta and norm} and \eqref{I: eq: theta*2}.

Moreover, in \cite[Theorem 1 and Remark 6]{CFMP} the domain of analyticity of the semigroup on $L^{2}(\gamma_{\infty})$ was characterised. In particular it was proved that the angle $\theta_{2}$ is optimal; i.e. $(T(t))_{t>0}$ does not have a 
bounded analytic extension to any sector larger than $\bS_{\theta_{2}}$. Namely we have the following result.

\begin{proposition}[\cite{CFMP}]\label{p: properties of L2} 
Let $\theta_{2}^*=\theta_{2}^*(B)$ be as above.
Then,
\begin{itemize}
\item[{\rm (i)}] $(T(t))_{t>0}$ extends to an analytic contraction semigroup on $L^{2}(\gamma_{\infty})$ in the sector $\bS_{\theta_{2}}$.
\item[{\rm (ii)}] If $(T(t))_{t>0}$ extends to a bounded analytic semigroup on $L^{2}(\gamma_{\infty})$ in the sector $\bS_{\theta}$, for some $\theta\in (0,\pi/2)$, then $\theta\leq\theta_{2}$.
\end{itemize}
\end{proposition}

\begin{remark}
It follows from \eqref{eq: theta*2} and Remark~\ref{r: symmetric OU} that $\theta_{2}=\pi/2$ if and only if $\oL_{2}$ is self-adjoint; see also \cite[Remark 3]{CFMP}.  
\end{remark}

We next turn to the analyticity properties of $T(t)$ on $L^r(\mu_\infty)$ for $1<r<\infty$, $r\ne2$. 

\begin{notation}
\label{n: angle}
Following \cite{CFMP}, for $1<r<\infty$ and with $\theta^{*}_{2}$ as on page \pageref{t2}, define
\begin{equation}
\label{eq: 1 theta r}
\theta^{*}_{r}=\theta^{*}_{r}(B)=\arctan\frac{\sqrt{(r-2)^{2}+r^{2}(\tan\theta^{*}_{2})^{2}}}{2\sqrt{r-1}}\,.
\end{equation}
An equivalent way of introducing $\vartheta_r^*\in[0,\pi/2)$ is through the identity

\begin{equation}\label{eq: alternative}
\sin\vartheta_r=\sin\phi_r\sin\vartheta_2\,,
\end{equation}
where $\theta_{r}=\pi/2-\theta^{*}_{r}$ and $\phi_r=\arccos\mod{1-2/r}$. 
It follows directly from \eqref{eq: 1 theta r} or \eqref{eq: alternative} that $\theta^{*}_{r}=\theta^{*}_{r/(r-1)}$ for $1<r<\infty$. 
\end{notation}

The following result, proved by Chill, Fa\v{s}angov\'a, Metafune and Pallara, extends Proposition~\ref{p: properties of L2} to the case $r\neq 2$ and it is closely related to our paper.

\begin{proposition}[{\cite[Theorem~2 and Remark~6]{CFMP}}]\label{p: OU on Lp}
Suppose that $1<r<\infty$. Then,
\begin{itemize}
\item[{\rm (i)}] $(T(t))_{t>0}$ extends to an analytic contraction semigroup on $L^{r}(\gamma_{\infty})$ in the sector $\bS_{\theta_{r}}$.
\item[{\rm (ii)}] If $(T(t))_{t>0}$ extends to a bounded analytic semigroup on $L^{r}(\gamma_{\infty})$ in the sector $\bS_{\theta}$, for some $\theta\in (0,\pi/2]$, then $\theta\leq\theta_{r}$.
\end{itemize}
\end{proposition}
In the special case when $\oL_{2}$ is self-adjoint (by Remark~\ref{r: symmetric OU}, it corresponds to the case when $B=B^{*}=Q/2$), the semigroup $(T(t))_{t>0}$ is Markovian and $\theta^{*}_{2}=0$. Therefore, Proposition~\ref{p: OU on Lp} (i) is just a particular case of a more general result that holds for all symmetric contraction semigroups; see \cite[Th\'eor\`eme 3]{Bakry3} for the case of diffusion semigroups, \cite[Corollary 3.2]{lp} for the case of sub-Markovian semigroups, and \cite[Corollary 6.2]{KR} for the general case. See also \cite[Remark 34]{CD-mult} and \cite{Haase2, HKV}.

Note also that when $A=Q=I$, Proposition~\ref{p: OU on Lp} is a consequence of a more precise result by Epperson \cite{E}.

\begin{remark}\label{r: optimal}
Proposition~\ref{p: OU on Lp} gives (see e.g. \cite[Proposition~3.4.4]{Haase} or \cite[Section~4 in Chapter~II]{EN} for details) that for $1<r<\infty$,
\begin{equation}\label{eq: omegar=}
\omega(\oL_{r})=\theta^{*}_{r}. 
\end{equation}
\end{remark}

\subsection*{Bounded $H^\infty$-calculus}
To the best of our knowledge, in the case when $\oL$ is nonsymmetric, the following are the best results concerning  the bounded $H^\infty$-calculus of $\oL_{r}$, $1<r<\infty$, that can be recovered from the existing literature.

First notice that $\oL_{r}$ is a sectorial operator that is not one-to-one, but this can be easily fixed since ${\rm N}(\oL_{r})=\{{\rm constant\ functions }\}$ (see e.g. \cite[Theorem~8.1.16]{lorenzi}) and the projection $\oP_{r}$ onto the null space is given by the formula
\begin{equation}\label{eq: projection}
\oP_{r} f=\int_{\R^{n}}f\wrt\gamma_{\infty},\quad f\in L^{r}(\gamma_{\infty}),\quad 1<r<\infty.
\end{equation}
Therefore, for every $r\in (1,\infty)$, we have that $\overline{{\rm R}}(\oL_{r})=L^{r}_{0}(\gamma_{\infty})=\{f\in L^{r}(\gamma_{\infty}): \int_{\R^{n}} f\wrt\gamma_{\infty}=0\}$, and $\oL_{r}$ is a sectorial one-to-one operator with dense range on $L^{r}_{0}(\gamma_{\infty})$. Hence the functional calculus devised in \cite{Mc, CDMY} applies to our case.

\begin{lemma}[Bounded functional calculus for $\oL_{2}$]
\label{c: fc L2}
We have that 
$$
\omega_{H^{\infty}}(\oL_{2})=\omega(\oL_{2})=\theta^{*}_{2},
$$
and 
$$
\|m(\oL_{2})f\|_{2}\leq \left(2+
2/\sqrt3\right)\|m\|_{\theta}\|f\|_{2},\quad f\in L^{2}_{0}(\gamma_{\infty}),
$$
for every $\theta\in (\theta^{*}_{2},\pi)$ and for all $m\in H^{\infty}(\bS_{\theta})$.
\end{lemma}
\begin{proof}
By combining inequality \eqref{eq: *} with \eqref{eq: omegar=} for $r=2$, we obtain $\omega_{H^{\infty}}(\oL_{2})\geq\omega(\oL_{2})=\theta^{*}_{2}$. The corollary now follows    by combining Proposition~\ref{p: properties of L2} (i) with a result of Crouzeix and Delyon \cite{CrDe} (see also \cite[Corollary~7.1.17]{Haase}). 
\end{proof}

Since $(T(t))_{t>0}$ is a positivity preserving contraction semigroup on $L^{r}(\gamma_{\infty})$ for any  $1<r<\infty$, one may combine Lemma~\ref{c: fc L2} with the transference technique of Coifman-Weiss \cite{CW,CRW} and with a complex interpolation argument in order to study bounded $H^{\infty}$-calculus of $\oL_{r}$, $1<r<\infty$. 
This was pointed out by Cowling \cite{cowling} in the symmetric case and by 
Duong \cite{Duong} in the nonsymmetric one.

\begin{proposition}
\label{p: fc alla cowling}
Suppose that $1<r<\infty$, $r\neq 2$. Then 
$$
\omega_{H^{\infty}}(\oL_{r})\leq \theta^{*}_{2}+\theta_{2}\mod{1-2/r}.
$$
More precisely, if
$
\theta^{*}_{2}+\theta_{2}\mod{1-2/r}<\theta<\pi,
$
then there exists $C(r,\theta)>0$ such that
$$
\|m(\oL_{r})f\|_{p}\leq C(r,\theta)\|f\|_{\theta}\|f\|_{r}
$$
for all $m\in H^{\infty}(\bS_{\theta})$ and $f\in L^{r}_{0}(\gamma_{\infty})$.
\end{proposition}
\begin{proof}
Let $1<s<\infty$. Since $(T(t))_{t>0}$ is a strongly continuous, contractive and positivity preserving semigroup on $L^{s}_{0}(\gamma_{\infty})$, by Coifman-Weiss transference technique  $\oL_{s}$ has a bounded $H^{\infty}(\bS_{\theta'})$-calculus, for every $\theta'>\pi/2$ \cite{CW, CRW, cowling, Duong}. In particular, for every $\theta'>\pi/2$,
$$
\norm{\oL^{iu}_{s}f}{s}\leq C(s,\theta')e^{\theta'|u|}\norm{f}{s},\quad f\in L^{s}_{0}(\gamma_{\infty}),\quad u\in\R.
$$
Lemma~\ref{c: fc L2} implies that, for every $\theta^{''}>\theta^{*}_{2}$,
$$
\norm{\oL^{iu}_{2}f}{2}\leq \left(2+
2/\sqrt3\right)e^{\theta^{''}|u|}\norm{f}{2},\quad f\in L^{2}_{0}(\gamma_{\infty}),\quad u\in\R.
$$
By interpolating the two estimates above, for every $\theta>\theta^{*}_{2}+\theta_{2}\mod{1-2/r}$,
$$
\norm{\oL^{iu}_{r}f}{r}\leq C(r,\theta)e^{\theta|u|}\norm{f}{p},\quad f\in L^{r}_{0}(\gamma_{\infty}),\quad u\in\R.
$$
The proposition now follows from \cite[Theorem 5.4]{CDMY}.
\end{proof}

\begin{remark}
If one is just interested in proving that $\omega_{H^{\infty}}(\oL_{r})<\pi/2$, instead of the argument above one can use \cite[Corollary 5.2 and Theorem 5.3]{KW} (see also \cite[Lemma~8.4]{MV}).
\end{remark}

\subsection*{Main theorem in the finite dimensional case}
One of the principal aims of this paper is to improve Proposition~\ref{p: fc alla cowling} by obtaining the sharp angle in the bounded holomorphic functional calculus for the finite dimensional Ornstein-Uhlenbeck operator $\oL_{r}$, $1<r<\infty$. This is our result:

\begin{theorem}
\label{t: fc principal}
For each $r\in (1,\infty)$ let $\theta^{*}_{r}$ be the angle defined in \eqref{eq: 1 theta r}. Then,
$$
\omega_{H^{\infty}}(\oL_{r})=\omega(\oL_{r})=\theta^{*}_{r}.
$$ 
Moreover, for every $\theta>\theta^{*}_{r}$ there exists $C>0$ which depends only on $r$, $\theta$ and $\theta^{*}_{2}$, such that
$$
\|m(\oL_{r})f\|_{r}\leq C\|m\|_{\theta}\|f\|_{r},\quad f\in L^{r}_{0}(\gamma_{\infty}),
$$
for all $m\in H^{\infty}(\bS_{\theta})$.
\end{theorem}
The proof of Theorem~\ref{t: fc principal} is postponed to Section~\ref{s: proof of mult thm}.  Note that a simple calculation based on \eqref{eq: 1 theta r} or \eqref{eq: alternative} shows that
$\theta^{*}_{r}<\theta^{*}_{2}+\theta_{2}\mod{1-2/r}$ for $r\in(1,\infty)\backslash\{2\}$, i.e., Theorem~\ref{t: fc principal} indeed improves Proposition \ref{p: fc alla cowling}.

\begin{remark}
When  $\oL_2$ is symmetric (i.e. when $\theta^{*}_{2}=0$) Theorem~\ref{t: fc principal} is a particular case of a universal multiplier result previously proved by the two authors of the present paper \cite[Theorem 1]{CD-mult}.
In the special case when $Q=A=I$, Theorem~\ref{t: fc principal} is a consequence of \cite{fcou, MMS}, apart from the fact that the estimates of the norms in \cite{fcou, MMS} depend on the dimension. 
\end{remark}

\begin{remark}
\label{r: analiticity}
One could study bounded $H^{\infty}$-calculus for degenerate Ornstein-Uhlenbeck operators; i.e. Ornstein-Uhlenbeck operators whose diffusion part is associated with a nonnegative and symmetric quadratic form $Q$ with ${\rm N}(Q)\neq \{0\}$.
We will show in Section~\ref{I: s: OU} and Section~\ref{I: s: proof of mult thm} that an analogue of Theorem~\ref{t: fc principal} holds in the degenerate case under the assumption that the associated Ornstein-Uhlenbeck semigroup is analytic on $L^{2}(\gamma_{\infty})$ in some sector of positive angle; see Theorem~\ref{I: t: fc principal}.
Note however that degenerate Ornstein-Uhlenbeck semigroups may fail to be analytic \cite{Fu,GvN}; if this is the case then clearly $\omega_{H^{\infty}}(\oL_{r})\geq\omega(\oL_{r})\geq\pi/2$.
Note also that when ${\rm N}(Q_{\infty})=\{0\}$, by a result of B. Goldys \cite[Corollary 2.6]{Goldys} the nondegeneracy condition ${\rm N}(Q)=\{0\}$ is necessary for the analyticity of the associated finite dimensional Ornstein-Uhlenbeck semigroup.
\end{remark}

\section{Bounded $H^{\infty}$-calculus via Nazarov-Treil Bellman function}
\label{s: Nazarov-Treil}
Unless specified otherwise, we assume everywhere in this section that $p\geqslant 2$ and $q=p/(p-1)$. Fix $\delta>0$. The Bellman function we use is the function $\cQ=\cQ_{p,\delta}:~\R^{2}\times\R^{2}\longrightarrow\R_+$ defined by 
\begin{equation}\label{eq: defi Bellman}
\cQ(\zeta,\eta)= 
|\zeta|^p+|\eta|^{q}+\delta
\left\{
\aligned
& |\zeta|^2|\eta|^{2-q} & ; & \ \ |\zeta|^p\leqslant |\eta|^q\\
& \frac{2}{p}\ |\zeta|^{p}+\left(\frac{2}{q}-1\right)|\eta|^{q}
& ; &\ \ |\zeta|^p\geqslant |\eta|^q\,.
\endaligned\right.
\end{equation}

The origins of $\cQ$ lie in the paper of F. Nazarov and S. Treil \cite{NT}.
A modification of their function has been later applied by A. Volberg and the second author in \cite{DV,DV-Sch}. Here we use a simplified variant which comprises only two variables. It was introduced in \cite{DV-Kato} and used by the present authors in \cite{CD, CD-mult}.

The construction of the original Nazarov--Treil function in \cite{NT} was one of the earliest examples of the so-called Bellman function technique, which was introduced in harmonic analysis shortly beforehand by Nazarov, Treil and Volberg \cite{NTV}.  
The name ``Bellman function'' stems from the stochastic optimal control, see \cite{NTV1} for details. The same paper \cite{NTV1} explains the connection between the Nazarov--Treil--Volberg approach and the earlier work of Burkholder on martingale inequalities, see \cite{Bu1} and also \cite{Bu2,Bu4}. For an in-depth treatise on recent advances in martingale inequalities the reader is referred to \cite{Os}.
If interested in the genesis of Bellman functions and the overview of the method, the reader is also referred to \cite{V,NT,W}.
The method has seen a whole series of applications, yet until recently (see \cite{CD, CD-mult}) mostly in Euclidean harmonic analysis.

In the course of the last few years, the Nazarov--Treil function considered here was found to possess nontrivial properties \cite{DV-Sch,CD-mult, CD} that reach much beyond the need for which it had been originally constructed in \cite{NT}.
In the present paper we continue the exploration of the convexity properties of $\cQ$ (see Theorem~\ref{t: convexity} and Remark~\ref{r: gen convex}).

\medskip
It is a direct consequence of the definition above that the function $\cQ$ belongs to $C^1(\R^4)$, and is of order $C^2$ everywhere {\it except} on the set
$$
\Upsilon_0=\mn{(\zeta,\eta)\in  \R^{2}\times\R^{2}}{(\eta=0)\vee (|\zeta|^p=|\eta|^q)}\,.
$$
For $\zeta,\eta\in\R^{2}$ write $\zeta=(\zeta_1,\zeta_2)$, $\eta=(\eta_1,\eta_2)$ and define
$$
\partial_\zeta=\frac{1}{2}\left(\partial_{\zeta_1}-i\partial_{\zeta_2}\right)\quad
{\rm and}\quad \partial_\eta=\frac{1}{2}\left(\partial_{\eta_1}-i\partial_{\eta_2}\right).
$$
The following estimates are also a straightforward consequence of the definition of $\cQ$.

\begin{proposition}\label{p: estimates Bellman}
For every $(\zeta,\eta)\in\R^{2}\times\R^{2}$ we have
$$
0\leqslant \cQ(\zeta,\eta)\leqslant (1+\delta)\left(|\zeta|^p+|\eta|^q\right),
$$
and
\begin{align*}
&2|(\partial_{\zeta}\cQ)(\zeta,\eta)|\leqslant (p+2\delta)\max\{|\zeta|^{p-1},|\eta|\},\\
&2|(\partial_{\eta}\cQ)(\zeta,\eta)|\leqslant (q+(2-q)\delta)|\eta|^{q-1}.
\end{align*}
\end{proposition}
\begin{remark}
It is sometime useful to think of $\cQ$ as a function defined on $\C\times\C$, by using the canonical identification of $\R^{2}$ with $\C$. This fact will be often implicitly used in this paper.
\end{remark}

We now state an abstract result, which is an extension to the nonsymmetric case of an analogous technique used by present authors for proving a universal multiplier theorem for generators of symmetric contraction semigroups \cite[Section 4, Remark 34]{CD-mult}.  

\begin{theorem}\label{t: heat flow}
Let $(\Omega,\mu)$ be a $\sigma$-finite measure space, and let $\oA$ be a closed, densely defined and one-to-one operator on $L^{p}(\mu)$. Let $\delta\in (0,1)$ and let $\cQ$ be the Bellman function associated with $\delta$. Suppose that there exist $\theta\in [0,\pi/2)$ and $C_{0}>0$ such that
\begin{equation}
\label{eq: heat flow infinitesimal}
C_{0}\mod{\int_{\Omega}\oA f\cdot \overline{g}\wrt\mu}
\leq 
2\Re\int_\Omega \left(e^{\pm i\theta}(\partial_\zeta \cQ)(f,g)\oA f+ e^{\mp i\theta}(\partial_\eta \cQ)(f,g)\oA^{*} g\right)\wrt\mu,
\end{equation}
for all $f\in {\rm D}(\oA)$ and every $g\in{\rm D}(\oA^{*})$. Then,
\begin{itemize}
\item[{\rm(i)}] $-\oA$ is the generator of an analytic contraction semigroup on $L^{p}(\mu)$ in the sector $\overline{\bS}_{\theta}$. 
\item[{\rm(ii)}] $-\oA^{*}$ is the generator of an analytic contraction semigroup on $L^{q}(\mu)$ in the sector $\overline{\bS}_{\theta}$. 
\item[{\rm(iii)}] For every $\theta^{'}>\theta^*$, the operator $\oA$ has bounded $H^{\infty}(\bS_{\theta'})$-calculus on $L^{p}(\mu)$. That is, there exists $C_{1}=C_{1}(p,\theta',C_{0})>0$ such that for all $m\in H^{\infty}(\bS_{\theta'})$,
$$
\|m(\oA)f\|_{p}\leq C_{1}\|m\|_{\theta'}\|f\|_{p},\quad f\in L^{p}(\mu).
$$
\end{itemize}
\end{theorem}
\begin{proof}
It follows from \eqref{eq: defi Bellman} that $Q(\zeta,0)=|\zeta|^p$ and $Q(0,\eta)=|\eta|^q$, up to some positive multiplicative constants. Thus, by taking separately $g=0$ and $f=0$ in \eqref{eq: heat flow infinitesimal}, we get
$$
\Re\Big(e^{\pm i\theta}\int_{\Omega} 
\bar f|f|^{p-2} \oA f\,\wrt\mu\Big) 
\geq0
\hskip 30pt \text{and} \hskip 30pt 
\Re\Big(e^{\pm i\theta}\int_{\Omega} 
\bar g|g|^{q-2} \oA^* g\,\wrt\mu\Big) 
\geq0\,.
$$
Items (i) and (ii) are now a consequence of the well-known Lumer-Phillips theorem; see e.g. \cite[Corollary~4.4]{pazy}.

In order to prove (iii), we combine the complex-time-heat-flow technique, developed in \cite[Section 4]{CD-mult} by the present authors, with a result by Cowling, Doust, McIntosh and Yagi \cite{CDMY} which relates functional calculus for a sectorial operator with bilinear estimates involving the semigroup generated by the sectorial operator.

It follows from (i) that $\oA$ is (one-to-one and) sectorial of angle $\omega(\oA)\leq\theta^*$, see e.g. \cite[Proposition 3.4.4]{Haase}. Denote by $(T(t))_{t>0}$ the semigroup generated by $-\oA$ on $L^{p}(\mu)$. Then $(T^{*}(t))_{t>0}$ is the semigroup generated by $-\oA^{*}$ on $L^{q}(\mu)$. Therefore, by \cite[Theorem 4.6 and Example 4.8]{CDMY}, part (iii) will follow once that we have proved the following bilinear estimate, 
\begin{equation}\label{eq: fc bilinear}
\int^{\infty}_{0}\mod{\int_{\Omega}\oA T(te^{\pm i\theta})(u) {\overline{T^{*}(te^{\mp i\theta})(v)}} \wrt\mu}\wrt t\leq C_{2}\|u\|_{p}\|v\|_{q},
\end{equation}
for all $u\in L^{p}(\mu)$, for every $v\in L^{q}(\mu)$ and for some $C_{2}=C_{2}(p,C_{0})>0$.

\medskip
For the purpose of proving \eqref{eq: fc bilinear} we apply the heat-flow technique developed in \cite{CD-mult}. Fix $u\in L^{p}(\mu)$, $v\in L^{q}(\mu)$ and  consider the functional
$$
\cE(t)=\int_{\Omega}\cQ\left(T(te^{\pm i\theta})(u),T^{*}(te^{\mp i\theta})(v)\right)\wrt\mu,\quad t\geq 0.
$$
Estimates of Proposition~\ref{p: estimates Bellman} ensure that $\cE$ is continuous on $[0,\infty)$, differentiable on $(0,\infty)$ with a continuous derivative and
\begin{equation*}
\label{eq: flow der}
\cE'(t)=\int_\Omega \frac{\pd}{\pd  t}\cQ(T(te^{\pm i\theta})(u),T^{*}(te^{\mp i\theta})(v))\wrt\mu.
\end{equation*}
A straightforward calculation shows that, for every $t>0$,
$$
\aligned
-\cE'(t)
=2\Re\int_\Omega &\Big[e^{\pm i\theta}\left(\partial_\zeta \cQ\right)(T(te^{\pm i\theta})(u),T^{*}(te^{\mp i\theta})(v))\cdot\oA T(te^{\pm i\theta})(u)\\
&+ e^{\mp i\theta}\left(\partial_\eta \cQ\right)(T(te^{\pm i\theta})(u),T^{*}(te^{\mp i\theta})(v))\cdot\oA^{*} T^{*}(te^{\mp i\theta})(v)\Big]\wrt\mu.
\endaligned
$$
By (i) and (ii), for every $t>0$ we have $T(te^{\pm i\theta})(u)\in \Dom(\oA)$ and $T^{*}(te^{\mp i\theta})(v)\in \Dom(\oA^{*})$. Therefore,  it follows from the assumption \eqref{eq: heat flow infinitesimal} that
\begin{equation*}
-\cE'(t)\geq C_{0}\mod{\int_{\Omega}\oA T(te^{\pm i\theta})(u) {\overline{T^{*}(te^{\mp i\theta})(v)}} \wrt\mu},\quad t>0.
\end{equation*}
By integrating from $0$ to $\infty$ both sides of the inequality above, and using the first estimate in Proposition~\ref{p: estimates Bellman}, we obtain
$$
C_{0}\int^{\infty}_{0}\mod{\int_{\Omega}\oA T(te^{\pm i\theta})(u) {\overline{T^{*}(te^{\mp i\theta})(v)}} \wrt\mu}\wrt t\leq (1+\delta)\left(\|v\|^{p}_{p}+\|v\|^{q}_{q}\right).
$$
The bilinear estimate \eqref{eq: fc bilinear} (hence the theorem) now follows by replacing $u$ with $k u$ and $v$ with $v/k$ and minimising with respect to $k>0$ the right-hand side of the inequality above.
\end{proof}
\begin{remark}
When $\theta=0$, similar heat-flow techniques corresponding to Bellman functions have so far been employed in the Euclidean case \cite{PV, NV, DV, DV-Sch, DV-Kato, DoPe} and recently also in the Riemannian case \cite{CD}. For a different perspective on heat-flow techniques, various examples and references we refer the reader to the papers by Bennett et al. \cite{Bennett, BCCT}.
\end{remark}

\section{Convexity of the Bellman function.} 
In this section we state the convexity result (Theorem~\ref{t: convexity}) that will be the principal ingredient of the proof of Theorem~\ref{t: fc principal}. The very same convexity result will be used again for proving the infinite dimensional analogue of Theorem~\ref{t: convexity} (see Theorem~\ref{I: t: fc principal}). For this reason we state Theorem~\ref{t: convexity} in rather general form. In order to do that we first  fix more notation.

Let $\cH$ be a real (possibly infinite dimensional) separable Hilbert space. Recall that we denoted by $\cH^\C$ its complexification. If $\xi=(\xi_{1},\xi_{2})\in \cH\times\cH$ we use the notation
$$
\tilde\xi=\xi_{1}+i\xi_{2}\in \cH^\C.
$$
We write $\xi_1=\Re\tilde\xi$ and $\xi_2=\Im\tilde\xi$.

If $D\in\cB(\cH^\C)$, denote respectively by $\Re D$ and $\Im D$ the real and the imaginary part of $D$. Note that $\Re D,\Im D\in\cB(\cH)$; by abuse of notation we use the same symbol for denoting their complexifications, so that we write
$$
D=(\Re D)+i(\Im D).
$$

We introduce the following bounded operator matrix acting on $\cH\times\cH=\R^{2}\otimes \cH$
$$
\cM(D)=\left[
\begin{array}{rr}
\Re D & -\Im D\\
\Im D  & \Re D
\end{array}
\right]\,.
$$
Observe that $\cM(D^*)=\cM(D)^{*}$ and $\cM(DE)=\cM(D)\cM(E)$, for all $D,E\in\cB(\cH^\C)$, and that
\begin{equation}\label{eq: sesqui*}
\sk{\cM(D)\alpha}{\beta}_{\cH\times\cH}=\Re\sk{D\tilde\alpha}{\tilde\beta}_{\cH^\C},\quad \forall\alpha,\beta\in \cH^\C.
\end{equation}
\begin{notation}
\label{d: H^{M}}
Let $\cH$ be a real separable Hilbert space. Suppose that $D,E\in \cB(\cH^\C)$, $\Psi:\R^{2}\rightarrow \R$ and $\Phi:\R^{4}\rightarrow \R$. For all $s\in\R^{2}$ and $v\in\R^{4}$, and for every  $\xi=(\xi_{1},\xi_{2})\in \cH^2$ and  $\omega=(\alpha_{1},\alpha_{2},\beta_{1},\beta_{2})\in \cH^4$, we set
$$
H^{D}_{\Psi}[s;\xi]= \sk{({\rm Hess}(\Psi;s)\otimes I_{\cH})\xi}{\cM(D)\xi}_{\cH^2}
$$
and
$$
 H^{(D,E)}_{\Phi}[v;\omega]=
 \sk{({\rm Hess}(\Phi;v)\otimes I_{\cH})\omega}{\left[\cM(D)\oplus \cM(E)\right]\omega}_{\cH^4}.
$$
In block notation,
$$
H^{D}_{\Psi}[s;\xi]=
\Sk{ {\rm Hess}(\Psi;s)
\left[\begin{array}{c}\xi_1 \\\xi_2\end{array}\right]}
{\left[\begin{array}{cr}\Re D & -\Im D \\ \Im D & \Re D\end{array}\right]
\left[\begin{array}{c}\xi_1 \\ \xi_2\end{array}\right]
}_{\cH^2}
$$
and
$$
H^{(D,E)}_{\Phi}[v;\omega]=
\Sk{ {\rm Hess}(\Phi;v) 
\left[\begin{array}{c}\alpha_1 \\ \alpha_2 \\ \beta_1 \\ \beta_2\end{array}\right]}
{\left[\begin{array}{crcr}\Re D & -\Im D &  &  \\ \Im D & \Re D&  &  \\ &  & \Re E &-\Im E  \\ &  & \Im E & \Re E \end{array}\right] 
\left[\begin{array}{c}\alpha_1 \\ \alpha_2 \\ \beta_1 \\ \beta_2\end{array}\right]
}_{\cH^4}.
$$
\end{notation}
Note that for every $s\in \R^{2}$ and $\xi=(\xi_{1},\xi_{2})\in \cH^2$, the map $\cB(\cH^\C)\rightarrow \R$, defined by 
$D\mapsto H^{D}_{\Psi}[s;\xi]$, 
is $\R$-linear. In particular, for $R\in\cB(\cH)$ and $\theta\in [0,\pi/2]$ we have
\begin{equation}\label{eq: R-linear}
H^{e^{\pm i\theta}R}_{\Psi}[s;\xi]=\cos\theta \cdot H^{R}_{\Psi}[s;\xi]\pm \sin\theta \cdot H^{iR}_{\Psi}[s;\xi].
\end{equation}
If $R\in \cB(\cH)$ one has 
$
H^{R}_{\Psi}[s;\xi]=H^{R^{*}}_{\Psi}[s;\xi],
$
so that $H^{R}_{\Psi}[s;\xi]=H^{R_{{\bf s}}}_{\Psi}[s;\xi]$. In particular, if $R\in\cB(\cH)$ is accretive (i.e. if $R_{{\bf s}}$ is nonnegative) we have
\begin{equation}\label{eq: adjoint}
\aligned
H^{R}_{\Psi}[s;\xi]
&=H^{I}_{\Psi}\left[s;\left(R^{1/2}_{{\bf s}}\xi_{1},R^{1/2}_{{\bf s}}\xi_{2}\right)\right]\\
&=\Sk{{\rm Hess}(\Psi;s)\cdot\left[\begin{array}{c}R^{1/2}_{{\bf s}}\xi_1 \\ R^{1/2}_{{\bf s}}\xi_2\end{array}\right]}    {\left[\begin{array}{c}R^{1/2}_{{\bf s}}\xi_1 \\ R^{1/2}_{{\bf s}}\xi_2\end{array}\right]}_{\cH^2}.
\endaligned
\end{equation}

Let $\{e_{j}\}^{\infty}_{j=1}$ be an orthonormal base of $\cH$. Then for all $s\in\R^{2}$ and  $\xi=(\xi_{1},\xi_{2})\in \cH^2$ one has 
\begin{equation}
\label{eq: tensor}
H^{I}_{\Psi}\left[s;\xi\right]=\sum^{\infty}_{j=1}\Big\langle {\rm Hess}(\Psi;s)\cdot\left[\begin{array}{c}\xi_{1,j} \\ \xi_{2,j}\end{array}\right] ,   \left[\begin{array}{c}\xi_{1,j} \\ \xi_{2,j}\end{array}\right]\Big\rangle_{\R^{2}},
\end{equation}
where $\xi_{i}=\sum^{\infty}_{j=1}\xi_{i,j}e_{j}$ for $i=1,2$.\\
 
\subsection*{Numerical range angle}
Fix a real separable Hilbert space $\cH$. 
Suppose the operator $B\in\cB(\cH)$ is {\it strictly accretive}; i.e. suppose that $\sk{B\xi}{\xi}\geq\lambda\nor{\xi}^2$ for some $\lambda>0$ and all $\xi\in \cH$ \cite{K}.   
Clearly, its complexification $B^\C$ is then strictly accretive as well.

By a small abuse of notation, we will henceforth denote the complexification $B^\C$ of $B$ just by $B$.

Let $\theta^{*}_{2}=\theta^{*}_{2}(B)$ be the angle of the smallest closed sector in $\C_{+}$ contaning the numerical range of $B$; i.e. $\theta^{*}_{2}$ denotes the smallest angle in $ [0,\pi/2)$ for which 
\begin{equation}\label{eq: N2 bis}
|\Im\sk{B\xi}{\xi}_{\cH^\C}|\leq \tan\theta^{*}_{2}\cdot\Re\sk{B\xi}{\xi}_{\cH^\C},\quad \forall \xi\in \cH^\C.
\end{equation}
From \eqref{eq: sim and antisim} and the strict accretivity of $B$ it follows that 
its symmetric part $B_{{\bf s}}$ is invertible and positive definite, and 
$$
\tan\theta^{*}_{2}=\sup \left\{\left|\Sk{B^{-1/2}_{{\bf s}}B_{{\bf a}}B^{-1/2}_{{\bf s}}\xi}{\xi}_{\cH^\C}\right|:\ \norm{\xi}{\cH^\C}=1\right\}.
$$ 
The operator $B^{-1/2}_{{\bf s}}B_{{\bf a}}B^{-1/2}_{{\bf s}}$ is normal, therefore we may conclude that
\begin{equation}
\label{eq: theta and norm}
\tan\theta^{*}_{2}=\norm{B^{-1/2}_{{\bf s}}B_{{\bf a}}B^{-1/2}_{{\bf s}}}{\cB\left(\cH^\C\right)}.
\end{equation}
See, for example, \cite[Theorem 1.4-2]{GR} or \cite[Theorem 12.25]{RU2}.

For $1<r<\infty$, define $\theta^{*}_r\in(0,\pi/2]$ by means of \eqref{eq: 1 theta r}. We are now ready to state the convexity result for the Bellman function $\cQ$ defined in \eqref{eq: defi Bellman}.

\begin{theorem}
\label{t: convexity} 
Let $B\in\cB(\cH)$ and $\theta^{*}_2=\theta^{*}_2(B)$ be as above. Fix $p\in[2,\infty)$. For every $0\leq\theta<\theta_{p}$ there exist $\delta=\delta(p,\theta_{2},\theta)\in(0,1)$ and $a_{0}=a_{0}(p,\theta_{2},\theta)>0$ such that, if $\cQ$ is the Bellman function \eqref{eq: defi Bellman} associated with $\delta$, and
$$
C\in\{e^{i\theta}B, e^{-i\theta}B, e^{i\theta}B^{*},e^{-i\theta}B^{*}\},
$$
then
\begin{equation}\label{eq: final}
H^{(C,C^{*})}_{\cQ}[v;\omega]\geq a_{0}\cdot\Nor{B_{{\bf s}}^{1/2}\tilde\alpha}_{\cH^\C}\Nor{B_{{\bf s}}^{1/2}\tilde\beta}_{\cH^\C},
\end{equation}
for all $v\in\R^{4}\setminus \Upsilon_{0}$ and $\omega=(\alpha,\beta)\in \cH^2\times \cH^2$. 
\end{theorem}
The proof of Theorem~\ref{t: convexity} is postponed to Section~\ref{s: proof of convexity thm}.

\begin{remark}\label{r: gen convex}
Theorem~\ref{t: convexity} is a generalisation of  \cite[Theorem~15]{CD-mult}. Indeed, the quantity $\cR_{\phi}(Q)[v;\omega]$ defined in \cite[equation (18)]{CD-mult} is nothing but $H^{e^{i\phi}I_{\C},e^{-i\phi}I_{\C}}_{\cQ}[v;\omega]$. Moreover, when $B=B^{*}=I_{\C}$ we have that $\theta^{*}_{2}=0$, so that $\theta_{p}=\phi_p=\arccos|1-2/p|$. Therefore, \cite[Theorem~15]{CD-mult} corresponds to Theorem~\ref{t: convexity} in the particular case when $B=B^{*}=I_{\C}$. Note also that Theorem~\ref{t: convexity} can be considered as an extension of \cite{DV-Kato}, where A. Volberg and the second author of the present paper proved the analogue of \eqref{eq: final} for $\theta=0$, $C\in\cB(\R^{n})$ and $C^{*}$ replaced by $C$.  
\end{remark}

\subsection*{Regularisation}
Denote by $*$ convolution in $\R^{4}$, and let $(\psi_\epsilon)_{\epsilon>0}$ be a nonnegative, smooth and compactly supported approximation to the identity in $\R^4$. Since $\cQ\in C^1(\R^4)$ and its second-order partial derivatives exist on $\R^4\setminus \Upsilon_0$ and are locally integrable in $\R^4$, for all $C\in\cB(\cH^\C)$
\begin{equation}\label{eq: convolution}
H^{(C,C^{*})}_{\cQ*\psi_{\epsilon}}[v;\omega]=\int_{\R^4}H^{(C,C^{*})}_{\cQ}[v-v';\omega]\psi_\epsilon(v')\wrt v',
\end{equation}
for every $r>0$, $v\in\R^4$ and every $\omega=(\omega_1,\omega_2)\in \cH^2\times \cH^2$; see \cite[Th\'eor\`eme~V]{Schwartz} and \cite[Theorem~2.1]{Eberle}.

\begin{corollary}
\label{c: convexity}
Let $\epsilon>0$. Under the assumptions of Theorem~\ref{t: convexity},
\begin{equation*}
H^{(C,C^{*})}_{\cQ*\psi_{\epsilon}}[v;\omega]
\geq a_{0}(p,\theta_{2},\theta)\cdot \Nor{B_{{\bf s}}^{1/2}\tilde\alpha}_{\cH^\C}\Nor{B_{{\bf s}}^{1/2}\tilde\beta}_{\cH^\C},
\end{equation*}
for all $v\in\R^4$ and $\omega=(\alpha,\beta)\in \cH^2\times \cH^2$.
\end{corollary}
\begin{proof}
The corollary immediately follows from \eqref{eq: convolution} and Theorem~\ref{t: convexity}.
\end{proof}

\section{Proof of Theorem~\ref{t: convexity}}\label{s: proof of convexity thm} 
Recall that $\cH$ is a separable real Hilbert space, $B$ is a bounded strictly accretive operator on $\cH$ and $B_{\bf s}$ the symmetric part of $B$. Recall also that by a small abuse of notation, we denote the complexification $B^\C$ of $B$ just by $B$.

Take $p\geq 2$, $q=p/(p-1)$ and a parameter $\delta>0$ that will be fixed later. It is convenient to rewrite the Bellman function defined in \eqref{eq: defi Bellman} as a linear combination of tensor products of power functions. For $r\geq 0$ define the power function $F_r: \R^{2} \rightarrow \R_+$ by the rule
$$
F_{r}(s)=|s|^{r}.
$$
Let $\bena$ denote the constant function of value $1$ on $\C$; i.e. $\bena=F_0$. Then
\begin{equation*}
\cQ= \begin{cases}
(1+2\delta/p) F_p\otimes\bena +[1+\delta(1-2/p)]\bena\otimes F_q,& {\rm if }\quad |\zeta|^p\geq |\eta|^q,\\
\ &\\
F_p\otimes\bena+\bena\otimes F_q+\delta F_2\otimes F_{2-q}, & {\rm if }\quad |\zeta|^p\leq |\eta|^q.
\end{cases}
\end{equation*}
\label{eq: Bellman with powers} 
Therefore, by Notation~\ref{d: H^{M}},
\begin{equation}
\label{eq: HQ}
H^{(C,C^{*})}_{\cQ}[v;\omega]= 
\begin{cases}
(1+2\delta/p) H^{C}_{F_p}[\zeta;\alpha] +[1+\delta(1-2/p)]H^{C^{*}}_{F_q}[\eta;\beta],& {\rm if }\quad |\zeta|^p> |\eta|^q>0,\\
\ &\\
H^{C}_{F_p}[\zeta;\alpha]+H^{C^{*}}_{F_q}[\eta;\beta]+\delta H^{(C,C^{*})}_{F_2\otimes F_{2-q}}[v;\omega], & {\rm if }\quad |\zeta|^p< |\eta|^q,
\end{cases}
\end{equation}
for all $\omega=(\alpha,\beta)\in \cH^2\times \cH^2$.

\subsection*{Convexity of power functions.}
Note that ${\rm Hess}(F_{2};s)=2I_{\R^{2}}$, for all $s\in\R^{2}$. Hence, for every $\xi=(\xi_{1},\xi_{2})\in \cH^2$, one has   
\begin{align*}
2\Re\sk{B\tilde\xi}{\tilde\xi}_{\cH^\C}=H^{B}_{F_{2}}[s;\xi]
\hskip 30pt \text{and}\hskip 30pt 
2\Im\sk{B\tilde\xi}{\tilde\xi}_{\cH^\C}=-H^{iB}_{F_{2}}[s;\xi] .
\end{align*}
Therefore, by \eqref{eq: N2 bis}, for all $s\in \R^{2}$ and every $\xi\in \cH^2$,
\begin{equation}
\label{eq: Nr of B}
|H^{iB}_{F_{2}}[s;\xi]|\leq \cot\theta_{2}\cdot H^{B}_{F_{2}}[s;\xi].
\end{equation}
Moreover, the very same estimate holds with $B$ replaced by $B^{*}$.

\medskip
We now show that an estimate analogous to \eqref{eq: Nr of B} holds when $F_{2}$ is replaced by $F_{r}$, $r>1$, and $\theta_{2}$ is replaced by $\theta_{r}$, $r>1$. It turns out that this is just a reformulation of \cite[Theorem 2]{CFMP} and \cite[Theorem 3.4]{MvN} (see also \cite[Theorem 1.1]{CFMP2}).

\begin{proposition}
\label{p: convexity of Fr}
For all $\xi\in \cH^2$ and every $s\in \R^{2}$, if $r\geq 2$, or every $s\in \R^{2}\setminus\{0\}$, if $1<r<2$, we have 
\begin{equation*}
|H^{iB}_{F_{r}}[s;\xi]|\leq \cot\theta_{r}\cdot H^{B}_{F_{r}}[s;\xi].
\end{equation*}
Moreover, the very same estimate holds with $B$ replaced by $B^{*}$.
\end{proposition}
\begin{proof}
Fix an orthonormal base $\{e_{j}\}^{\infty}_{j=1}$of $\cH$ and for each $n\in\N$ denote by $\cH_{n}$ the finite dimensional subspace of $\cH$ spanned by $\{e_{1},\dots,e_{n}\}$. Denote by $P_{n}$ the orthogonal projection of $\cH$ onto $\cH_{n}$ and set $B_{n}=P_{n}B|_{\cH_n}\in\cB(\cH_{n})$. For each $n\in\N$ identify $\cH_{n}$ with $\R^{n}$ through the canonical map 
$$
\cI_n:\sum^{n}_{j=1}u_{j}e_{j}\mapsto (u_{1},\dots,u_{n}).
$$
For each $k\in\N$ identify $\C^k$ with $\R^{2k}$ through the map $\cV:\C^k\rightarrow\R^{2k}$ defined by
$$
\cV(z_1,\hdots,z_k)=(\Re z_1,\hdots,\Re z_n,\Im z_1,\hdots,\Im z_k).
$$
If $f\in C^{1}(\R^{n},\C)$, then a simple calculation gives  
\begin{align}
r\Re\sk{B_{n}\nabla f}{\nabla(f|f|^{r-2})}_{\C^{n}}&=H^{B_{n}}_{F_{r}}[\cV (f);\cV(\nabla f)],\label{eq: 1 *}\\
-r\Im\sk{B_{n}\nabla f}{\nabla(f|f|^{r-2})}_{\C^{n}}&=H^{iB_{n}}_{F_{r}}[\cV (f);\cV(\nabla f)],\label{eq: 2 *}
\end{align}
where the two identities above have to be understood to hold everywhere in $\R^{n}$ if $r\geq 2$, and everywhere on $\{f(x)\neq 0\}$ if $1<r<2$.

Fix $z\in\C$, $\xi\in \cH^\C$ and $n\in\N$. 
Then $P_n\xi\equiv(P_n(\Re\xi),P_n(\Im\xi))\in \cH_n\times \cH_n$, which identifies with $\R^n\times\R^n$ through $\cI_n$ and furthermore with $\C^n$ through $\cV$. Choose $f\in C^{1}(\R^{n},\C)$ in a way such that $f(0)=z$ and $\nabla f(0)=P_n\xi$. Combine \eqref{eq: 1 *} and \eqref{eq: 2 *} with the calculations in \cite[Proof of Theorem 1.1]{CFMP2} (see also \cite[Theorem 4.1]{MvN}) and finally pass to the limit as $n\rightarrow \infty$. The outcome is
\begin{align*}
&H^{B}_{F_{r}}[\cV(z);\cV(\xi)]
=r|z|^{r-4}\left(
\Sk{B_{{\bf s}}\Im(\overline{z}\xi)}{\Im(\overline{z}\xi)}
+(r-1)
\Sk{B_{{\bf s}}\Re(\overline{z}\xi)}{\Re(\overline{z}\xi)}_{\cH}\right),\\
&H^{iB}_{F_{r}}[\cV(z);\cV(\xi)]
=-r|z|^{r-4}\Sk{\left[(r-2)B_{{\bf s}}+rB_{{\bf a}}\right]\Im(\overline{z}\xi)}{\Re(\overline{z}\xi)}_{\cH}.
\end{align*}
(The bottom line corrects an insignificant sign misprint occurred in \cite{CFMP2}.) 

Recall that, by assumptions on $B$, the symmetric part $B_{{\bf s}}$ is strictly accretive. Define
$$
x=x(z,\xi)=B^{1/2}_{{\bf s}}\Re(\overline{z}\xi);\quad y=y(z,\xi)=B^{1/2}_{{\bf s}}\Im(\overline{z}\xi).
$$
Then,
\begin{align*}\label{eq: formula meta}
&H^{B}_{F_{r}}[\cV(z);\cV(\xi)]=r|z|^{r-4}\left[\nor y^{2}_{\cH}+(r-1) \nor x^{2}_{\cH}\right],\\
&-H^{iB}_{F_{r}}[\cV(z);\cV(\xi)]=r|z|^{r-4}\Sk{ \left[(r-2)I+rB^{-1/2}_{{\bf s}}B_{{\bf a}}B^{-1/2}_{{\bf s}}\right]y}{x}_{\cH}.
\end{align*}
Therefore, if we set
$$
T=(r-2)I+rB^{-1/2}_{{\bf s}}B_{{\bf a}}B^{-1/2}_{{\bf s}},
$$
we obtain
$$
\left|H^{iB}_{F_{r}}[\cV(z);\cV(\xi)]\right|\leq r|z|^{r-4}\nor T\cdot \norm{y}{\cH}\cdot  \norm{x}{\cH}\leq\frac{\|T\|}{2\sqrt{r-1}}H^{B}_{F_{r}}[\cV(z);\cV(\xi)].
$$
It remains to estimate the norm of $T$ as an operator acting on the Hilbert space $\cH^\C$. Since $B^{-1/2}_{{\bf s}}B_{{\bf a}}B^{-1/2}_{{\bf s}}$ is antisymmetric, by using \eqref{eq: theta and norm} we deduce that
$$
\|T\|^{2}=(r-2)^{2}+\Nor{B^{-1/2}_{{\bf s}}B_{{\bf a}}B^{-1/2}_{{\bf s}}}^{2}=(r-2)^{2}+(\tan\theta^{*}_{2})^{2},
$$
so the proposition follows for $B$. The fact that it also holds for $B^{*}$ follows, by repeating the proof above, from the fact that by \eqref{eq: N2 bis} we have $\theta^{*}_{2}(B^{*})=\theta^{*}_{2}(B)$.
\end{proof}

For every $r>1$ and $\theta\geq 0$ set 
$$
\Delta(r,\theta)=\frac{\sin(\theta_r-\theta)}{\sin\theta_r}.
$$
Recall that $\theta_{q}=\theta_{p}$, so that $\Delta(q,\theta)=\Delta(p,\theta)$.

Until the end of this section it will be convenient to write
$$
S=B_{\bf s}^{1/2}.
$$

\begin{lemma}
\label{l: domination H}
Let $\theta\in[0,\pi/2]$ and $\xi\in \cH\times\cH$. Then for all $s\in\R^{2}$ if $r\geq 2$, or for all $s\in\R^{2}\setminus\{0\}$ if $1<r<2$, we have
$$
H^{e^{\pm i\theta}B}_{F_{r}}[s;\xi]\geq \Delta(r,\theta) H^{B}_{F_{r}}[s;\xi],
$$
and the same estimate holds with $B$ replaced by $B^{*}$.
\end{lemma}
\begin{proof}
The lemma rapidly follows by combining \eqref{eq: R-linear}, applied with $R=B$ or $R=B^{*}$ and $\Psi=F_{r}$, with Proposition~\ref{p: convexity of Fr}.
\end{proof}

\begin{lemma}
\label{l: H Fr}
Let $\xi=(\xi_{1},\xi_{2})\in \cH\times\cH$. Then for all $s\in\R^{2}$ if $r\geq 2$, or for all $s\in\R^{2}\setminus\{0\}$ if $1<r<2$, we have
\begin{align*}
 H_{F_r}^{B}[s;\xi]\geq\min\{1,r-1\} r|s|^{r-2}\nor{S\tilde\xi}^2
   \end{align*}
 and the same estimate holds with $B$ replaced by $B^{*}$.
\end{lemma}
\begin{proof}
By \cite[Lemma~19]{CD-mult} applied with $\phi=0$, we have
$$
{\rm Hess}(F_{r};s)=\frac{r^2}2|s|^{r-2}\cD_{r,0}(s),
$$
where $\cD_{r,0}(s)$ is a real symmetric matrix with
$$
 \det \cD_{r,0}(s)=\left(\frac{r}{2}\right)^{2}\left(1-\mod{1-\frac{2}{r}}^{2}\right)
 \quad {\rm and}\quad 
 {\rm tr\,}\cD_{r,0}(s)=r.
$$
Consequently, in the sense of quadratic forms on $\R^{2}$,
\begin{equation}
\label{eq: Hess Fr}
{\rm Hess}(F_{r};s)\geq \min\{1,r-1\}r|s|^{r-2}I_{\R^{2}}\,.
\end{equation}
The lemma now follows by combining  \eqref{eq: Hess Fr} with \eqref{eq: adjoint} and \eqref{eq: tensor} applied with $\Psi=F_{r}$ and $R=B$ or $R=B^{*}$.
\end{proof}

\begin{corollary}
\label{c: H Fr}
Suppose that $1<r<\infty$, $0\leq \theta\leq\theta_{r}$
and $\xi=(\xi_{1},\xi_{2})\in \cH\times \cH$. 
Let 
$$
C\in\left\{e^{i\theta}B, e^{-i\theta}B, e^{i\theta}B^{*},e^{-i\theta}B^{*}\right\}.
$$
Then, for all $s\in\R^{2}$ if $r\geq 2$, or for all $s\in\R^{2}\setminus\{0\}$ if $1<r<2$, we have
$$
H^{C}_{F_{r}}[s;\xi]\geq 
\min\{1,r-1\}  r\Delta(r,\theta)|s|^{r-2}\nor{S\tilde\xi}^2.
$$
\end{corollary}

\begin{proof}
The corollary rapidly follows by combining Lemma~\ref{l: domination H} with Lemma~\ref{l: H Fr}.
\end{proof}

\begin{lemma}
\label{l: H F2-r}
Suppose that $1<r<2$ and $D\in \cB(\cH^\C)$. Then, for every $s\in \R^{2}\setminus\{0\}$ and  $\xi=(\xi_{1},\xi_{2})\in \cH^{2}\times\cH^{2}$,
$$
 H_{F_{2-r}}^{D}[s;\xi]
 = -2(r-1)|s|^{-r}\Re\langle D\tilde{\xi},\tilde{\xi}\rangle_{\cH^\C}+|s|^{2-2r}H^{D}_{F_{r}}[s;\xi]. 
$$
\end{lemma}
\begin{proof}
An easy computation (see \cite[eq. $(28)$]{CD-mult}) gives,
$$
{\rm Hess }(F_{2-r};s)=-2(r-1)|s|^{-r}I_{\R^{2}}+|s|^{2-2r}{\rm Hess}(F_{r};s).
$$
Therefore,
$$
H^{D}_{F_{2-r}}[s;\xi]=-2(r-1)|s|^{-r}\sk{\cM(D)\xi}{\xi}_{\cH^2}
+|s|^{2-2r}H^{D}_{F_{r}}[s;\xi]
$$
and the lemma now follows from \eqref{eq: sesqui*}.
\end{proof}

\begin{corollary}
\label{c: H F2-r}
 Suppose that $1<r<2$ and $0\leq\theta\leq \theta_{r}$. Let 
 $$
 C\in\left\{e^{i\theta}B, e^{-i\theta}B, e^{i\theta}B^{*},e^{-i\theta}B^{*}\right\}.
 $$
Then for every $s\in\R^{2}\setminus\{0\}$ and $\xi\in \cH\times \cH$,
$$
 H_{F_{2-r}}^{C}[s;\xi]\geq -2(r-1)\left(1+\tan\theta^{*}_{2}\right) |s|^{-r}\nor{S\tilde\xi}^2.
 $$ 
\end{corollary}
\begin{proof}
By Corollary~\ref{c: H Fr}, $ H_{F_{r}}^{C}[s;\xi]\geq0$. Therefore, by Lemma~\ref{l: H F2-r} applied with $D=C$, 
$$
 H_{F_{2-r}}^{C}[s;\xi]\geq -2(r-1)|s|^{-r}\Re\langle C\tilde{\xi},\tilde{\xi}\rangle_{\cH^\C}\,.
$$
Now \eqref{eq: N2 bis} and the first identity in \eqref{eq: sim and antisim} imply that

\[
\Re\sk{C\tilde{\xi}}{\tilde{\xi}}_{\cH^\C}\leq  |\sk{B\tilde{\xi}}{\tilde{\xi}}_{\cH^\C}|\leq(1+\tan\theta^{*}_{2})\nor{S\tilde\xi}^2.
\qedhere
\]

\end{proof}

\begin{lemma}
\label{l: F2tensorF2-r}
Suppose that $1<r<2$. Let $D,E\in\cB(\cH^\C)$. Then for every $v=(\zeta,\eta)\in\R^{2}\times(\R^{2}\setminus\{0\})$ and for all $\omega=(\alpha,\beta)\in \cH^{2}\times \cH^{2}$, we have
\begin{align}\label{eq: F2tensorF2-r}
H^{(D,E)}_{F_{2}\otimes F_{2-r}}[v;\omega]&=F_{2-r}(\eta)H^{D}_{F_{2}}[\zeta;\alpha]+F_{2}(\zeta)H^{E}_{F_{2-r}}[\eta;\beta]\\
&\hskip 30pt+2(2-r)|\eta|^{-r}\langle(\eta\cdot\beta)\zeta,\cM(D)\alpha\rangle_{\cH^2}\nonumber\\
&\hskip 30pt+2(2-r)|\eta|^{-r}\langle(\zeta\cdot\alpha)\eta,\cM(E)\beta\rangle_{\cH^2},\nonumber
\end{align}
where
$$
\eta\cdot\beta=\eta_{1}\beta_{1} +\eta_{2}\beta_{2}\in \cH
\hskip 30pt {\it and }\hskip 30pt (\eta\cdot\beta)\zeta=\big((\eta\cdot\beta)\zeta_{1},(\eta\cdot\beta)\zeta_{2}\big)\in \cH^2.
$$
\end{lemma}
\begin{proof}
The lemma follows from the definition of $H^{(D,E)}_{F_{2}\otimes F_{2-r}}[v;\omega]$, and the identity
\[
\pd^2_{\zeta_j\eta_k}(F_{2}\otimes F_{2-r})(\zeta,\eta)=2(2-q)\zeta_j\eta_k|\eta|^{-r}, 
\hskip 30pt \text{for } j,k=1,2.
\qedhere
\]
\end{proof}

\begin{corollary}
\label{c: F2tensorF2-r}
Suppose that $1<r<2$ and that $0\leq\theta\leq\theta_{r}$. Let 
$$
C\in\left\{e^{i\theta}B, e^{-i\theta}B, e^{i\theta}B^{*},e^{-i\theta}B^{*}\right\}.
$$
Then, for every $v=(\zeta,\eta)\in\R^{2}\times\R^{2}$ with $|\zeta|^{r/(r-1)}<|\eta|^{r}$, and $\omega=(\alpha,\beta)\in \cH^2\times \cH^2$,
$$
\aligned
H^{(C,C^{*})}_{F_{2}\otimes F_{2-r}}[v;\omega]\geq 2\Delta(2,\theta)|\eta|^{2-r}\Nor{S\widetilde\alpha}_{\cH^\C}^{2}
&-2(r-1)(1+\tan\theta^{*}_{2}) |\eta|^{r-2}\nor{S\tilde\beta}_{\cH^\C}^{2}\\
&-4(2-r)(1+\tan\theta^{*}_{2})\nor{S\tilde\alpha}_{\cH^\C}\nor{S\tilde\beta}_{\cH^\C}.
\endaligned
$$
\end{corollary}
\begin{proof}
We apply Lemma~\ref{l: F2tensorF2-r} with $D=C$ and $E=C^{*}$. In order to estimate the first two terms in the right-hand side of \eqref{eq: F2tensorF2-r} we use Corollary~\ref{c: H Fr} with $r=2$ (note that $\theta\leq\theta_{r}<\theta_{2}$) and Corollary~\ref{c: H F2-r}. 

We now estimate the last two terms in the right hand side of \eqref{eq: F2tensorF2-r}. By \eqref{eq: sesqui*},
\begin{align*}
\Sk{(\eta\cdot\beta)\zeta}{\cM(C)\alpha}_{\cH^2}
&=\Re\Sk{C\tilde{\alpha}}{\widetilde{(\eta\cdot\beta)\zeta}}_{\cH^\C}
\hskip 4.6pt\leq \left|\Sk{B\tilde{\alpha}}{\widetilde{(\eta\cdot\beta)\zeta}}_{\cH^\C}\right|\\
\langle(\zeta\cdot\alpha)\eta,\cM(C^{*})\beta\rangle_{\cH^2}
&=\Re\Sk{C^{*}\tilde{\beta}}{\widetilde{(\zeta\cdot\alpha)\eta}}_{\cH^\C}
\leq \left|\Sk{B\tilde{\beta}}{\widetilde{(\zeta\cdot\alpha)\eta}}_{\cH^\C}\right|\,.
\end{align*}
By \eqref{eq: N2 bis} and the generalised Cauchy-Schwarz inequality \cite[Proposition 1.8]{O} we have
\begin{equation}
\label{vsenoschnoe}
|\sk{Bz}{w}_{\cH^\C}|\leq(1+\tan\theta^{*}_{2})\Nor{Sz}\Nor{Sw}.
\end{equation}
for $z,w\in\cH^\C$. One quickly sees that 
$$
\Nor{S\widetilde{(\eta\cdot\beta)\zeta}}_{\cH^\C}\leq
|\zeta||\eta|\nor{S\tilde\beta}_{\cH^\C}
$$
and similarly for $(\zeta\cdot\alpha)\eta$. The corollary now follows.
\end{proof}

\begin{proof}[{\bf Proof of Theorem~\ref{t: convexity}}]
Suppose that $|\zeta|^{p}>|\eta|^{q}>0$. Then $|\zeta|^{p-2}>|\eta|^{2-q}$, and for every $\delta>0$ the inequality \eqref{eq: final} with $a_{0}=2\Delta(p,\theta)$ follows from \eqref{eq: HQ} and Corollary~\ref{c: H Fr} applied first with $r=p$ and then with $r=q$.

Suppose now that $|\zeta|^{p}<|\eta|^{q}$. Then by \eqref{eq: HQ} and Corollary~\ref{c: H Fr} applied with $r=p$,
$$
H^{(C,C^{*})}_{\cQ}[v;\omega]\geq H^{C^{*}}_{F_q}[\eta;\beta]+\delta H^{(C,C^{*})}_{F_2\otimes F_{2-q}}[v;\omega].
$$
We now apply Corollary~\ref{c: H Fr} and Corollary~\ref{c: F2tensorF2-r}, both with $r=q$. To conclude it is now enough to take $\delta>0$ sufficiently small (with respect to $q$, $\theta_{2}$ and $\theta$) so that the following $2\times 2$ matrix is positive definite: 
$$
\left[\begin{array}{cc} \delta\Delta(2,\theta)
&\hskip 75pt 2\delta(1+\tan\theta^{*}_{2})(q-2)\\
2\delta(1+\tan\theta^{*}_{2})(q-2)
&\quad\quad \{q\Delta(q,\theta)-2\delta(1+\tan\theta^{*}_{2})\}(q-1)/2
\end{array}\right].
$$
With this choice of $\delta$ inequality \eqref{eq: final} holds with
\begin{align*}
& a_{0}=a_{0}(p,\theta_{2},\theta)=\sqrt{2\delta q(q-1)}\Delta(q,\theta).
\qedhere
\end{align*}
\end{proof}

\section{Proof of Theorem~\ref{t: fc principal}}\label{s: proof of mult thm}
Set $B=Q_{\infty}A^{*}$. Inequality \eqref{eq: *} together with Remark~\ref{r: optimal} imply that 
$$
\omega_{H^{\infty}}(\oL_{r})\geq\omega(\oL_{r})=\theta^{*}_{r}.
$$

Now let us prove the opposite inequality. For $f\in C^{\infty}_{c}(\R^{n})$ set
\begin{equation}\label{eq: div form}
\oL f=\nabla^{*}_{\infty}(B\nabla f)\quad {\rm and}\quad\oL^{*}f=\nabla^{*}_{\infty}(B^{*}\nabla f).
\end{equation}
We stress the fact that $\oL^{*}$ is the Ornstein-Uhlenbeck operator associated with a sesquilinear form of the type \eqref{eq: sesquilinear form}, but with $B$ replaced by $B^{*}$.
Recall that  $\oL_s$ and $\oL^*_s$ denote the realisation on $L^s(\mu_\infty)$ of $\oL$ and $\oL^*$, respectively. By means of the theory of sesquilinear forms \cite{K,O}, $(\oL_{2})^{*}=\oL^{*}_{2}$. It follows that 
$$
(\oL_{r})^{*}=\oL^{*}_{r/(1-r)},\quad 1<r<\infty.
$$ 
Therefore, it is enough to prove that Theorem~\ref{t: fc principal} holds for all $r=p\geq 2$ for both $\oL_{p}$ and $\oL^{*}_{p}$. 

Fix $p>2$ and $0\leq \theta<\theta_{p}$. By Theorem~\ref{t: heat flow}, it suffices to prove that there exist $\delta>0$ and a corresponding Bellman function $\cQ$ defined by \eqref{eq: defi Bellman} such that \eqref{eq: heat flow infinitesimal} holds for the two one-to-one operators $\oL_{p}(I-\oP_{p})$ and $\oL^{*}_{p}(I-\oP_{p})$, where $\oP_{p}$ is the projection onto the null space defined by \eqref{eq: projection}.
\begin{lemma}\label{l: measure theory}
Let $(X,\mu)$ be a finite measure space. Let $r\in [1,\infty]$ and let $r'=r/(r-1)$. Suppose that $(F_{n})$ and $(G_{n})$ are two sequences of measurable functions with the following properties.
\begin{itemize}
\item[{\rm (i)}] $F_{n}\in L^{r}(\mu)$ and converges to $F$ in $L^{r}(\mu)$.
\item[{\rm (ii)}] There exists $C_{0}>0$ such that $\sup_{n}\|G_{n}\|_{r'}\leq C_{0}$.
\item[{\rm (iii)}] $G_{n}$ converges almost everywhere to a function $G$ which belongs to  $L^{r'}(\mu)$.
\end{itemize}
Then $(F_{n}G_{n})_{n\in\N}$ converges to $FG$ in $L^{1}(\mu)$.
\end{lemma}
\begin{proof}
Write
$$
F_{n}G_{n}-FG=(F_{n}-F)G+F(G-G_{n})\ca_{\{|G_{n}|\leq |G|+1\}}+(F\ca_{\{|G_{n}|>|G|+1\}}).(G-G_{n})
$$
Now use H\"older inequality and Lebesgue dominated convergence theorem. 
\end{proof}
Since $C^{\infty}_{c}(\R^{n})$ is a core for $\oL_{p}$ and $\oL^{*}_{p}$, and the partial derivatives of $\cQ$ satisfy the second estimate in Proposition~\ref{p: estimates Bellman}, by Lemma~\ref{l: measure theory} it is enough to show that  \eqref{eq: heat flow infinitesimal} holds for all $f,g\in C^{\infty}_{c}(\R^{n})$ when $\oA=\oL$ or $\oA=\oL^{*}$.

\medskip
Choose $\delta\in(0,1)$ and the corresponding $\cQ$ as in Theorem~\ref{t: convexity}. Fix $f,g\in C^{\infty}_{c}(\R^{n})$ and $\epsilon>0$. A straightforward integration by parts based on \eqref{eq: div form} gives

\begin{equation}
\label{eq: by parts}
\aligned
2\Re\int_{\R^n} &\left(e^{\pm i\theta}(\partial_\zeta \cQ*\psi_{\epsilon})(f,g)\oL f+ e^{\mp i\theta}(\partial_\eta \cQ*\psi_{\epsilon})(f,g)\oL^{*} g\right)\wrt\mu_\infty\\
&=\int_{\R^{n}}H^{\left(e^{\pm i\theta}B,e^{\mp i\theta}B^{*}\right)}_{\cQ*\psi_{\epsilon}}\left[(f,g);(\nabla f,\nabla g)\right]\wrt\gamma_{\infty}\,.
\endaligned
\end{equation}

It follows from Corollary~\ref{c: convexity}, applied with $\cH=\R^{n}$ and $B=Q_{\infty}A^{*}$, that for all $\epsilon>0$,
$$
\aligned
2\Re\int_{\R^n} \left(e^{\pm i\theta}(\partial_\zeta \cQ*\psi_{\epsilon})(f,g)\oL f+ e^{\mp i\theta}(\partial_\eta \cQ*\psi_{\epsilon})(f,g)\oL^{*} g\right)&\wrt\mu_\infty\\
&\hskip -90pt\geq a_{0}\int_{\R^{n}}\norm{B_{{\bf s}}^{1/2}\nabla f}{}\norm{B_{{\bf s}}^{1/2}\nabla g}{}\wrt\gamma_{\infty}\,.
\endaligned
$$

Since $\cQ\in C^{1}(\R^{4})$, one has that $\partial_\zeta \cQ*\psi_{\epsilon}$ and $\partial_\eta \cQ*\psi_{\epsilon}$ converge pointwise respectively to $\partial_\zeta \cQ$ and $\partial_{\eta}\cQ$ on $\R^{4}$, as $\epsilon\rightarrow 0_{+}$. Therefore, by the second estimate in Proposition~\ref{p: estimates Bellman} and Lebesgue dominated convergence theorem, we deduce that
$$
\aligned
2\Re\int_{\R^n} \left(e^{\pm i\theta}(\partial_\zeta \cQ)(f,g)\oL f+ e^{\mp i\theta}(\partial_\eta \cQ)(f,g)\oL^{*} g\right)&\wrt\mu_\infty\\
&\hskip -40pt\geq a_{0}\int_{\R^{n}}\norm{B_{\bf s}^{1/2}\nabla f}{}\norm{B_{\bf s}^{1/2}\nabla g}{}\wrt\gamma_{\infty}\,.
\endaligned
$$
By integrating by parts and using the identities \eqref{eq: div form} and \eqref{vsenoschnoe}, we see that
$$
\mod{\int_{\R^{n}}\oL f\cdot\overline{g}\wrt\gamma_{\infty}}\leq \int_{\R^{n}}|\sk{B\nabla f}{\nabla g}|\wrt\gamma_{\infty}
\leq(1+\tan\theta^{*}_{2})\int_{\R^{n}}\norm{B_{\bf s}^{1/2}\nabla f}{}\norm{B_{{\bf s}}^{1/2}\nabla g}{}\wrt\gamma_{\infty}.
$$
It follows that $(\oL,\R^{n},\gamma_{\infty})$ satisfies \eqref{eq: heat flow infinitesimal} for all $f,g\in C^{\infty}_{c}(\R^{n})$ with $C_{0}=a_{0}/(1+\tan\theta^{*}_{2})$.
By copying the same proof but with $B$ replaced by $B^*$, we see that the same conclusion holds also for $(\oL^{*},\R^{n},\gamma_{\infty})$, which finishes the proof of Theorem~\ref{t: fc principal}.\qed

\begin{remark}
The integration by parts \eqref{eq: by parts} was the main reason for introducing Notation~\ref{d: H^{M}} and studying Theorem~\ref{t: convexity}.
\end{remark}

\begin{remark}
Our proof of Theorem~\ref{t: fc principal} does not use  a priori the analyticity of Ornstein-Uhlenbeck semigroup proved in \cite[Theorem 2]{CFMP} (see Proposition~\ref{p: OU on Lp}). However, it is based on Theorem~\ref{t: convexity}, whose proof makes use of a calculation contained in \cite{CFMP} and \cite{CFMP2} that, in turn, is equivalent to \cite[Theorem 2]{CFMP} (see Proposition~\ref{p: convexity of Fr}).
\end{remark}

\section{Infinite dimensional setting}\label{I: s: OU}

The main sources for the background material presented in this section are \cite{vN, GvN, MvN, M} and the books \cite{Bogachev, Jan, Nu}. The interested reader should  also consult \cite{DPZ, DP, DPZ2}.

We consider a real separable Banach space $E$. We denote by $E^{*}$ its dual and by $\sk{x}{x^{*}}$, $x\in E$, $x^{*}\in E^{*}$ the dual paring.

We say that $Q\in\cB(E^{*},E)$ is nonnegative if $\sk{Qx^{*}}{x^{*}}\geq 0$, for all $x^{*}\in E^{*}$, and symmetric if $\sk{Qx^{*}}{y^{*}}=\sk{Qy^{*}}{x^{*}}$, for all $x^{*},y^{*}\in E^{*}$.
\subsection*{Gaussian measures on $E$ \cite{Bogachev}} For each $\sigma\geq 0$ the centered Gaussian measure $\gamma_{\sigma}$ of variance $\sigma^{2}$ in $\R$ is the Borel probability measure on $\R$ defined by $\gamma_{0}=\delta_{0}$ and
$$
\wrt\gamma_{\sigma}(t)=\frac{1}{\sqrt{2\pi}\sigma}\exp\left(-\frac{t^{2}}{2\sigma^{2}}\right),\quad t\in\R,\quad \sigma>0.
$$
A Borel probability measure $\gamma$ on $E$ is called (centred) Gaussian if for all  $x^{*}\in E^{*}$ the push forward $x^{*}_{\sharp}\gamma$ is a centred Gaussian measure on $\R$.

If $\gamma$ is a Gaussian measure on $E$, then the covariance operator $Q\in\cB(E^{*},E)$ associated with $\gamma$ is defined by the Bochner integral
$$
Qx^{*}=\int_{E}\sk{x}{x^{*}}x\wrt\gamma(x),\quad x^{*}\in E^{*}.
$$
By Fernique's theorem (see \cite[Corollary~2.8.6]{Bogachev}) Gaussian measures have finite moments of every order. It follows that $Q$ is well defined, nonnegative and symmetric. Moreover, for each $x^{*}\in E^{*}$ the measure $x^{*}_{\sharp}\gamma$ is of variance $\sk{Qx^{*}}{x^{*}}$.

\subsection*{Reproducing kernel Hilbert space \cite{GvN,vN}} 
Let $Q\in\cB(E^{*},E)$ be symmetric and nonnegative. Consider on $\Ran(Q)$ the bilinear form $\sk{\cdot}{\cdot}_{\cH_{Q}}$ defined by
$$
\sk{Qx^{*}}{Qy^{*}}_{\cH_{Q}}=\sk{Qx^{*}}{y^{*}},\quad x^{*},y^{*}\in E^{*}.
$$
The bilinear form $\sk{\cdot}{\cdot}_{\cH_{Q}}$ is well defined and induces on $\Ran(Q)$ a scalar product. The reproducing kernel Hilbert space (RKHS) $\cH_{Q}$ associated with $Q$ is the completion of $\Ran(Q)$ with respect to the norm induced by the inner product $\sk{\cdot}{\cdot}_{\cH_{Q}}$.

The identity map $I:\Ran(Q)\rightarrow E$ extends to a continuous embedding $i_{Q}:\cH_{Q}\hookrightarrow E$. Moreover, we have the factorisation $$i_{Q}i^{*}_{Q}=Q.$$

\subsection*{Hypothesis \Hmu.} Consider a symmetric and nonnegative operator $Q\in \cB(E^{*},E)$, and a strongly continuous semigroup $(S(t))_{t>0}$ on $E$ of generator $-A$.

By \cite[Proposition 1.2]{vN} the function $s\mapsto S(s)QS^{*}(s)x^{*}$ is strongly measurable. Therefore, for each $t>0$ we can define the symmetric nonnegative operator $Q_{t}\in \cB(E^{*},E)$ by the formula
$$
Q_{t}x^{*}=\int^{t}_{0}S(s)QS^{*}(s)x^{*}\wrt s,\quad x^{*}\in E^{*},
$$
where the integral exists as a Bochner integral on $E$.

When the semigroup $(S(t))_{t>0}$ is not uniformly exponentially stable, it might happen that the (weak) $\lim_{t\rightarrow\infty}Q_{t}x^{*}$ does not exist for all $x^{*}\in E^{*}$.

Moreover, in contrast to the finite dimensional case, it is not always true that every nonnegative symmetric operator in $\cB(E^{*},E)$ is the covariance of a Gaussian measure on $E$; for example if $E$ is a Hilbert space, then the class of covariance operators of Gaussian measures coincides with the class of all nonnegative symmetric operators of trace class on $E$. We refer the interested reader to \cite{Bogachev,GvN,vN} for several examples and comments.
\medskip

Following \cite{GvN,vN},  we say that Hypotesis \Hmu\ holds if 
\begin{itemize}
\item[{\rm (i)}] for all $x^{*}\in E^{*}$, the weak-$\lim_{t\rightarrow \infty}Q_{t}x^{*}$ exists in $E$;
\item[{\rm (ii)}] the symmetric nonnegative operator $Q_{\infty}\in\cB(E^{*},E)$ defined by
$$
Q_{\infty}x^{*}={\rm weak-}\lim_{t\rightarrow\infty}Q_{t}x^{*},\quad x^{*}\in E^{*}
$$
is the covariance of a Gaussian measure $\gamma_{\infty}$ on $E$.
\end{itemize}
It can be proved that \Hmu\ implies that for all $t>0$ the operator $Q_{t}$ is the covariance of a Gaussian measure on $E$ which we denote by $\gamma_{t}$; see, for example, \cite[Section~7]{GvN}.
\begin{notation}
We denote by $\cH_{\infty}$ the reproducing kernel Hilbert space associated with $Q_{\infty}$, and we set $i_{\infty}=i_{Q_{\infty}}$. 
\end{notation}

\begin{lemma}[{\cite[Theorem 6.2]{vN}}]
If (part (i) of ) hypothesis \Hmu\ holds, then the subspace $i_{\infty}(\cH_{\infty})$ is invariant for the action of the semigroup $(S(t))_{t>0}$ and the equation
$$
i_{\infty}\circ S_{\infty}(t)=S(t)\circ i_{\infty},\quad t>0,
$$
defines a strongly continuous contraction semigroup $(S_{\infty}(t))_{t>0}$ on $\cH_{\infty}$.
\end{lemma}
\begin{notation}
We denote by $-A_{\infty}$ the generator of $(S_{\infty}(t))_{t>0}$ on $\cH_{\infty}$.
\end{notation}
In the rest of this paper we will always assume \Hmu. Note however that some of the results that we will recall still hold true under weaker assumptions. The interested reader should consult \cite{vN,GvN} and the references contained therein. 

\subsection*{Paley-Wiener isometry} Note that $i^{*}_{\infty}(E^{*})$ is dense in $\cH_{\infty}$, because $i_{\infty}$ is injective.
The map $\phi:i^{*}_{\infty}x^{*}\mapsto \sk{\cdot}{x^{*}}$ extends to an isometry $\phi: \cH_{\infty}\rightarrow L^{2}_{\R}(\gamma_{\infty})$ \cite[Proposition~1.12]{M}. 
The map $\phi$ is called Paley-Wiener isometry. We will use the notation $\phi_{h}=\phi(h)$, $h\in \cH_{\infty}$.

\subsection*{Cylindrical functions}
If $H_{0}$ is a linear subspace of $H_{\infty}$ and  $k\in\N\cup\{\infty\}$, we denote by $\cF C^{k}_{b}(E;H_{0})$ the vector space of all, $\gamma_{\infty}$-almost everywhere defined, functions $f:E\rightarrow \C$ of the form
\begin{equation}\label{I: eq: cilindrica}
f(x)=\varphi(\phi_{h_{1}}(x),\dots,\phi_{h_{n}}(x)),
\end{equation}
where $n\geq 1$, $\varphi\in C^{k}_{b}(\R^{n})$, and $h_{1},\dots,h_{n}\in H_{0}$. If $H_{0}$ is a dense subspace of $H_{\infty}$, then $\cF C^{k}_{b}(E;H_{0})$, $k\in\N\cup\{\infty\}$, is dense in $L^{r}(\gamma_{\infty})$, whenever $1\leq r<\infty$; see, for example, \cite[Lemma~1.28]{M}.

\subsection*{Second quantisation}The Paley-Wiener isometry can be used to construct the Wiener-It\^{o} chaos decomposition of $L^{2}_{\R}(\gamma_{\infty})$ (see \cite[Chapter~2]{Jan} and \cite[Theorem~1.1.1]{Nu}), which identifies $L^{2}_{\R}(\gamma_{\infty})$ with the symmetric Fock space associated with $\cH_{\infty}$  \cite[Chapter~IV and Appendix~E]{Jan}. This identification allows us to associate with each contraction $T\in\cB(\cH_{\infty})$ a contraction $\Gamma(T)\in\cB(L^{2}_{\R}(\gamma_{\infty}))$, where $\Gamma$ denotes the real second quantisation functor. Analogously, we can associate with each contraction $U\in\cB(\cH^{\C}_{\infty})$ a contraction $\Gamma^{\C}(U)\in \cB(L^{2}(\gamma_{\infty}))$, where $\Gamma^{\C}$ denotes the complex second quantisation functor. For a contraction $T\in\cB(\cH_{\infty})$ one has $\Gamma(T)^{\C}=\Gamma^{\C}(T^{\C})$; see \cite[Appendix~E]{Jan}. The interested reader should consult also \cite{Ne,Si,CG,vN,M} and the references therein.

\subsection*{Lyapunov equation.}
The operator $Q_{\infty}$ is the minimal solution of the Lyapunov equation 
\begin{equation*}
Q_{\infty}A^{*}+AQ_{\infty}=Q;
\end{equation*}
meaning that $Q_{\infty}x^{*}\in \Dom(A)$ and $Q_{\infty}A^{*}x^{*}+AQ_{\infty}x^{*}=Qx^{*}$, for all $x^{*}\in \Dom(A^{*})$; see \cite[Theorem 4.4]{GvN}.

\subsection*{Ornstein-Uhlenbeck semigroup with invariant measure} 
The Ornstein-Uhlenbeck semigroup $(T(t))_{t>0}$ associated with $(S(t))_{t>0}$ and $Q$ is defined on $C_{b}(E)$ by the Kolmogorov's formula 
$$
T(t)f(x)=\int_{E}f(S(t)x+y)\wrt\gamma_{t}(y),\quad x\in E,\quad t\geq 0.
$$
The measure $\gamma_{\infty}$ is invariant for the Ornstein-Uhlenbeck semigroup. Moreover, for every $r\in [1,\infty]$, $(T(t))_{t>0}$ extends to a positivity preserving semigroup of contractions on $L^{r}(\gamma_{\infty})$, which is strongly continuous for $1\leq r<\infty$ and weak*-continuous for $r=\infty$; see \cite[Section~2]{vN}.

\begin{notation}
For $1<r<\infty$, denote by $-\oL_{r}$ the generator of $(T(t))_{t>0}$ on $L^{r}(\gamma_{\infty})$.
\end{notation}
It was proved in \cite[Theorem 1]{CG} (see also \cite[Theorem 6.12]{vN}) that on $L^{2}(\gamma_{\infty})$ the semigroup $(T(t))_{t>0}$ coincides with the complexification of the second quantisation of $(S^{*}_{\infty}(t))_{t>0}$; i.e.
\begin{equation}\label{I: eq: sqz*}
T(t)=\Gamma\left(S^{*}_{\infty}(t)\right)^{\C}=\Gamma^{\C}\left(S^{*}_{\infty}(t)^{\C}\right),\quad t>0.
\end{equation}
In particular, for every cylindrical function $f=\varphi (\phi_{h_{1}},\dots,\phi_{h_{n}})\in\cF C_{b}(E;\cH_{\infty})$ we have the Mehler's formula:
\begin{equation}\label{eq: Mehler inf}
\aligned
T(t)f(x)=\int_{E}\varphi\big(\phi_{S^{*}_{\infty}(t)h_{1}}(x)&+\phi_{\sqrt{I-S_{\infty}(t)S^{*}_{\infty}(t)}h_{1}}(y),\dots\\
&\dots,\phi_{S^{*}_{\infty}(t)h_{n}}(x)+\phi_{\sqrt{I-S_{\infty}(t)S^{*}_{\infty}(t)}h_{n}}(y) \big)\wrt\gamma_{\infty}(y).
\endaligned
\end{equation}
By \cite[Proposition 2.4]{GvN} the semigroup $(S^{*}_{\infty}(t))_{t>0}$ is strongly exponentially stable.

The infinite-dimensional analogue of \eqref{eq: projection} now follows from the continuity of the Paley-Wiener isometry, from \eqref{eq: Mehler inf} and from the density of cylindrical functions in $L^{r}(\gamma_{\infty})$: for every $r\in [1,\infty)$ the projection $\oP_{r}$ onto ${\rm N}(\oL_{r})$ is given by
\begin{equation}\label{I: eq: pro}
\oP_{r}f=\lim_{t\rightarrow\infty}T(t)f=\int_{E}f\wrt\gamma_{\infty},\quad f\in L^{p}(\gamma_{\infty}).
\end{equation}
\subsection*{Analyticity of the Ornstein-Uhlenbeck semigroup.} As we already remarked in Remark~\ref{r: analiticity}, the semigroup $(T(t))_{t>0}$ may fail to be analytic in $L^{2}(\gamma_{\infty})$ on  some sector of positive angle of the complex plane. It follows from \eqref{I: eq: sqz*} that $(T(t))_{t>0}$ has a bounded analytic extension in $L^{2}(\gamma_{\infty})$ on the sector $\bS_{\theta}$, $\theta>0$, if and only if $(S^{*}_{\infty}(t)^{\C})_{t>0}$ (or equivalently $(S_{\infty}(t)^{\C})_{t>0}$) extends to a contractive analytic semigroup on $\bS_{\theta}$; if this is the case, then $(T(t))_{t>0}$ is a contraction in $L^{2}(\gamma_{\infty})$ on $\bS_{\theta}$ \cite[Theorem 8.1]{GvN}.
\medskip 

Next result characterises analytic Ornstein-Uhlenbeck semigroups. 
\begin{proposition}[{\cite[Proposition 2.1 and Lemma 2.2]{MvN}}]\label{I: p: 2}
The following assertions are equivalent.
\begin{itemize}
\item[{(\rm i)}] The Ornstein-Uhlenbeck semigroup $(T(t))_{t>0}$ is analytic on $L^{2}(\gamma_{\infty})$ in some sector of positive angle.
\item[{(\rm ii)}] There exists a unique operator $B_{Q}\in\cB(H_{Q})$ such that 
$$
i_{Q}B_{Q}i^{*}_{Q}x^{*}=Q_{\infty}A^{*}x^{*},\quad x^{*}\in \Dom(A^{*}).
$$
In this case we have that
\begin{equation}
\label{I: eq: Lyapunov}
B_{Q}+B^{*}_{Q}=I.
\end{equation}
\end{itemize}
\end{proposition}
\noindent{\bf Assumption.} Being interested in holomorphic functional calculus, in the rest of this paper we will always assume that $(T(t))_{t>0}$ is analytic on $L^{2}(\gamma_{\infty})$, meaning that it has a contractive analytic extension on some sector of positive angle of the complex plane.

\subsection*{Directional gradient.} It can be shown \cite[Lemma 5.2]{GGvN} that ${\rm N}(i^{*}_{\infty})\subseteq {\rm N}(i^{*}_{Q})$. Therefore, the densely defined operator
\begin{equation}\label{I: eq: def of V}
V: i^{*}_{\infty}(E^{*})\subseteq H_{\infty}\rightarrow H_{Q}; \quad Vi^{*}_{\infty}x^{*}=i^{*}_{Q}x^{*},\quad x^{*}\in E^{*}
\end{equation}
is well defined. Moreover, since $(T(t))_{t>0}$ is analytic, $V$ is closable \cite[Proposition~8.7]{GvN}. We still denote by $V$ its closure. We have the following decomposition of $A^{*}_{\infty}$ \cite[ Theorem 2.16]{M}:
\begin{equation}\label{I: eq: A*infty}
A^{*}_{\infty}=V^{*}B_{Q}V,\quad \Dom(A^{*}_{\infty})=\{h\in\Dom(V): B_{Q}Vh\in\Dom(V^{*})\};
\end{equation}
i.e. $A^{*}_{\infty}$ is the operator associated with the bilinear form
$$
{\mathfrak b}(g,h)=\langle B_{Q}Vg,Vh\rangle_{\cH_{Q}},\quad \Dom({\mathfrak b})=\Dom(V).
$$
For each $r\in [1,\infty)$ the directional gradient in the direction of $V$ is the densely defined operator
$$
D_{V}:\cF C^{1}_{b}(E;\Dom(V))\subseteq L^{r}(\gamma_{\infty})\rightarrow L^{r}(\gamma_{\infty};\cH^{\C}_{Q}),
$$ 
defined by the rule 
\begin{equation}\label{I: eq: DV}
D_{V}f(x)=\sum^{n}_{j=1}\partial_{j}\varphi(\phi_{h_{1}}(x),\dots,\phi_{h_{n}}(x))\otimes Vh_{j},
\end{equation}
whenever $f$ is of the form \eqref{I: eq: cilindrica}.
\begin{notation} For every $r\in (1,\infty)$, the operator $D_{V}$ is closable; see \cite[Theorem 3.5]{GGvN} and \cite[Proposition~8.7]{GvN}. We still denote by $D_{V}$ its closure.  If $1<r<\infty$, denote by $\Dom_{r}(D_{V})$ the domain of $D_{V}$ as a closed operator from $L^{r}(\gamma_{\infty})$ to $L^{r}(\gamma_{\infty};\cH^{\C}_{Q})$, and by $D_{r}(D^{*}_{V})$ the domain of its adjoint. 
\end{notation}

\subsection*{Divergence form for the Ornstein-Uhlenbeck operator.} Recall that $-\oL_{r}$ denotes Ornstein-Uhlembeck operator on $L^{r}(\gamma_{\infty})$, $1<r<\infty$. We still denote by $B_{Q}$ the complexification of the operator in Proposition~\ref{I: p: 2}.
\begin{proposition}[{\cite[Theorem 2.3, Proposition~2.4]{MvN} and \cite[Lemma~4.8]{MV}}]
\label{I: p: 4}
The generator $\oL_{2}$ coincides with the operator associated with the densely defined, closed, continuous and accretive sesquilinear form
$$
{\mathfrak a}(f,g):=\int_{E}\sk{B_{Q}D_{V}f}{D_{V}g}_{\cH^\C_{Q}}\wrt\gamma_{\infty},\quad \Dom({\mathfrak a})=\Dom_{2}(D_{V}).
$$
In other terms, we have
\begin{equation}\label{I: eq: Divform}
\oL_{2}:=D^{*}_{V}B_{Q}D_{V}, \quad \Dom(\oL_{2})=\{f\in \Dom_{2}(D_{V}):\ B_{Q}D_{V}f\in \Dom_{2}(D^{*}_{V})\}.
\end{equation}
If $1<r<\infty$, then $\cF C^{\infty}_{b}(E;\Dom(A^{*}_{\infty}))$ is a core for $\oL_{r}$.
\end{proposition} 
The theory of sesquilinear forms \cite{K,O} implies that $(\oL_{2})^{*}$ and $A_{\infty}$ are associated, respectively, with the adjoint sequilinear form ${\mathfrak a^{*}}$ and the adjoint bilinear form ${\mathfrak b^{*}}$. Moreover, $(-\oL_{2})^{*}$ is the generator of the semigroup $(\Gamma^{\C}(S_{\infty}(t)^{\C}))_{t>0}$. The following result is now a consequence of \cite[Lemma~4.8]{MV}.

\begin{proposition}[\cite{MvN,MV}]\label{I: p: 4 dual}
The generator $(\oL_{2})^{*}$ coincides with the operator associated with the densely defined, closed, continuous and accretive sesquilinear form
$$
{\mathfrak a^{*}}(f,g):=\int_{E}\sk{B^{*}_{Q}D_{V}f}{D_{V}g}_{\cH^\C_{Q}}\wrt\gamma_{\infty},\quad \Dom({\mathfrak a^{*}})=\Dom_{2}(D_{V}).
$$
In other terms, we have
\begin{equation}\label{I: eq: Divform}
(\oL_{2})^{*}:=D^{*}_{V}B^{*}_{Q}D_{V}, \quad \Dom((\oL_{2})^{*})=\{f\in \Dom_{2}(D_{V}):\ B^{*}_{Q}D_{V}f\in \Dom_{2}(D^{*}_{V})\}.
\end{equation}
If $1<r<\infty$, then $\cF C^{\infty}_{b}(E;\Dom(A_{\infty}))$ is a core for $(\oL_{r})^{*}$. 
\end{proposition}
\subsection*{Sector of analyticity of the Ornstein-Uhlenbeck semigroup.} Assume \Hmu\ and suppose that the Ornstein-Uhlenbeck semigroup $(T(t))_{t>0}$ is analytic. Let $\theta^{*}_{2}=\theta^{*}_{2}(B_{Q})$ be the angle defined in \eqref{eq: N2 bis}, but with $B=B_{Q}$ and $\cH=\cH_{Q}$.

By combining \eqref{eq: theta and norm} with the Lyapunov equation \eqref{I: eq: Lyapunov}, we obtain that
\begin{equation}\label{I: eq: theta*2}
\theta^{*}_{2}=\theta^{*}_{2}(B_{Q})=\arctan\|B_{Q}-B^{*}_{Q}\|,
\end{equation}
and $\sigma(B_{Q})=1/2+\sigma((B_{Q})_{2})$, where $\sigma((B_{Q})_{2})\subset i\R$. By the spectral theorem for normal operators, the spectral radius of $(B_{Q})_{2}$ coincides with $\|B_{Q}-B^{*}_{Q}\|/2$. Therefore, $\theta^{*}_{2}$ coincides with the spectral angle of $B_{Q}$ on $\cH^\C_{Q}$; see also \cite[Remark 2]{CFMP}.
\begin{notation}
For $1<r<\infty$ define $\theta^{*}_{r}$ as in \eqref{eq: 1 theta r}. Recall the notation $\theta_{r}=\pi/2-\theta^{*}_{r}$.
\end{notation}
The following result extends to the infinite dimensional setting \cite[Theorem 2, Remark 6]{CFMP} (see Proposition~\ref{p: OU on Lp}) and removes the nondegeneracy assumption
on the diffusion operator $Q$.
\begin{proposition}[{\cite[Theorem~3.4 and Theorem~3.5]{MvN}}]\label{I: p: OU on Lp}
Suppose that $1<r<\infty$. Then,
\begin{itemize}
\item[{\rm (i)}] $(T(t))_{t>0}$ extends to an analytic contraction semigroup on $L^{r}(\gamma_{\infty})$ in the sector $\bS_{\theta_{r}}$.
\item[{\rm (ii)}] If $(T(t))_{t>0}$ extends to a bounded analytic semigroup on $L^{r}(\gamma_{\infty})$ in the sector $\bS_{\theta}$, for some $\theta\in (0,\pi/2)$, then $\theta\leq\theta_{r}$.
\end{itemize}
\end{proposition}
As a consequence of Proposition~\ref{I: p: OU on Lp} we have $\omega(\oL_{r})=\theta^{*}_{r}$, $1<r<\infty$.
\section{$H^{\infty}$-calculus for infinite dimensional Ornstein-Uhlenbeck operators}\label{I: s: proof of mult thm}
In this section we fix a real separable Banach space $E$, a nonnegative symmetric operator $Q\in\cB(E^{*},E)$ and a strongly continuous semigroup $(S(t))_{t>0}$ on $E$. We assume hypothesis \Hmu\ and we suppose that the Ornstein-Uhlenbeck semigroup $(T(t))_{t>0}$ associated with $Q$ and $(S(t))_{t>0}$ is analytic. Let $B_{Q}$ be the operator in Proposition~\ref{I: p: 2}, and let $\theta^{*}_{2}=\theta^{*}_{2}(B_{Q})$ be the angle defined in \eqref{I: eq: theta*2}. For every $r\in (1,\infty)$ let $\theta^{*}_{r}=\theta^{*}_{r}(B_{Q})$ be defined by means of \eqref{eq: 1 theta r}.

\begin{theorem}\label{I: t: fc principal}
Let $1<r<\infty$. Then, 
$$
\omega_{H^{\infty}}(\oL_{r})=\omega(\oL_{r})=\theta^{*}_{r}.
$$
Moreover, for every $\theta>\theta^{*}_{r}$ there exists $C>0$, which depends only on $r$, $\theta$ and $\theta^{*}_{2}$, such that
$$
\|m(\oL_{r})f\|_{r}\leq C\|m\|_{\theta}\|f\|_{r},\quad f\in L^{r}_{0}(\gamma_{\infty}),
$$
for all $m\in H^{\infty}(\bS_{\theta})$.
\end{theorem}
\begin{proof}
The proof of the theorem is an adaptation of that of Theorem~\ref{t: fc principal}. First notice that by \eqref{I: eq: pro} we have ${\overline {\rm R}(\oL_{r})}=L^{r}_{0}(\gamma_{\infty})$.
Recall that, by Proposition~\ref{I: p: 4} and Proposition~\ref{I: p: 4 dual}, the space $\cF C^{\infty}_{b}(E;\Dom(A^{*}_{\infty}))$ is a core for $\oL_{r}$ and $\cF C^{\infty}_{b}(E;\Dom(A_{\infty}))$ is a core for $\oL^{*}_{r}$.

We first prove the theorem for $r=p> 2$. By arguing as in the proof of Theorem~\ref{t: fc principal} in Section~\ref{s: proof of mult thm}, we see that it suffices to prove that there exists $\delta>0$ such that the corresponding Bellman function $\cQ$, defined by \eqref{eq: defi Bellman}, satisfies the integral condition \eqref{eq: heat flow infinitesimal} in Theorem~\ref{t: heat flow} with $\oA=\oL_{p}$ for all $f\in \cF C^{\infty}_{b}(E;\Dom(A^{*}_{\infty}))$, all $g\in\cF C^{\infty}_{b}(E;\Dom(A_{\infty}))$ and every $\theta<\theta_{p}$.\\

Choose $\delta\in(0,1)$ and the corresponding $\cQ$ as in Theorem~\ref{t: convexity}. 
Fix $\epsilon>0$, $f\in\cF C^{\infty}_{b}(E;\Dom(A^{*}_{\infty}))$ and $g\in\cF C^{\infty}_{b}(E;\Dom(A_{\infty}))$. Since $\partial_\zeta \cQ*\psi_{\epsilon},\partial_\eta \cQ*\psi_{\epsilon}\in C^{\infty}(\C^{2},\C)$ and since by \eqref{I: eq: A*infty} we have that $\Dom(A^{*}_{\infty})\subseteq\Dom(V)$ and $\Dom(A_{\infty})\subseteq \Dom(V)$, it follows that
$$
\partial_\zeta \cQ*\psi_{\epsilon}(f,g)\in\cF C^{\infty}_{b}(E;\Dom(V))\subseteq \Dom_{2}(D_{V}),\quad \partial_\eta \cQ*\psi_{\epsilon}(f,g)\in\cF C^{\infty}_{b}(E;\Dom(V))\subseteq \Dom_{2}(D_{V}).
$$
Therefore, by Proposition~\ref{I: p: 4} and Proposition~\ref{I: p: 4 dual}, the right hand side of \eqref{eq: heat flow infinitesimal} with $\cQ$ replaced by $\cQ*\psi_{\epsilon}$ and with $\oA=\oL_{2}$ can be rewritten as
$$
\Re\int_{E}\left\{\sk{e^{\pm i\theta}B_{Q}D_{V}f}{D_{V}[\partial_\zeta \cQ*\psi_{\epsilon}(f,g)]}_{\cH^\C_{Q}}+\sk{e^{\mp i\theta}B^{*}_{Q}D_{V}g}{D_{V}[\partial_\eta \cQ*\psi_{\epsilon}(f,g)]}_{\cH^\C_{Q}}\right\}\wrt\gamma_{\infty}\,.
$$
It follows from \eqref{I: eq: DV} and from Definition~\ref{d: H^{M}} that the sum of the real part of the inner products inside the integral above equals
$$
H^{\left(e^{\pm i\theta}B_{Q},e^{\mp i\theta}B^{*}_{Q}\right)}_{\cQ*\psi_{\epsilon}}\left[(f,g);(D_{V} f,D_{V} g)\right].
$$
We now apply Corollary~\ref{c: convexity} with $\cH=\cH_{Q}$ and $B=B_{Q}$, and we finish the proof of the theorem for $p>2$ exactly as in the proof of Theorem~\ref{t: fc principal} (see Section~\ref{s: proof of mult thm}). The theorem for $r=q=p/(p-1)<2$ follows by a duality argument.
\end{proof}

\subsection*{Acknowledgements}
The first author was partially supported by the Gruppo Nazionale per l'Analisi Matematica, la Probabilit\`a e le loro Applicazioni (GNAMPA) of the Istituto Nazionale di Alta Matematica (INdAM). The second author was partially supported by the Ministry of Higher Education, Science and Technology of Slovenia (research program Analysis and Geometry, contract no. P1-0291).

\bibliographystyle{amsxport}
\bibliography{biblio1}

\end{document}